\documentclass[10pt]{amsart}
\usepackage{latexsym,amsmath,amscd,amssymb,epsfig}
\usepackage{graphicx}
\usepackage{xcolor}

\theoremstyle{plain}
  \newtheorem{theorem}{Theorem}[section]
  \newtheorem{proposition}[theorem]{Proposition}
  \newtheorem{lemma}[theorem]{Lemma}
  \newtheorem{corollary}[theorem]{Corollary}
  
\theoremstyle{definition}
  \newtheorem{definition}[theorem]{Definition}
  \newtheorem{example}[theorem]{Example}
  
  \newtheorem{question}[theorem]{Question}
  \newtheorem{problem}[theorem]{Problem}
 \theoremstyle{remark}
  \newtheorem{remark}[theorem]{Remark}

\numberwithin{equation}{section}

\def\NN{{\mathbb N}}
\def\ZZ{{\mathbb Z}}

\def\RR{{\mathbb R}}

\def\LLL{{\mathcal{L}}}

\def\AAA{{\mathcal{A}}}
\def\BBB{{\mathcal{B}}}
\def\DDD{{\mathcal{D}}}
\def\EEE{{\mathcal{E}}}
\def\III{{\mathcal{I}}}
\def\JJJ{{\mathcal{J}}}
\def\NNN{{\mathcal{N}}}
\def\PPP{{\mathcal{P}}}
\def\ZZZ{{\mathcal{Z}}}

\def\Jconn{{\mathcal{J}_{\mathrm{conn}}}}

\def\mm{\mathfrak{m}}

\def\xx{{\mathbf{x}}}
\def\UU{{\mathbf{U}}}

\def\Des{\mathrm{Des}}
\def\des{\mathrm{des}}
\def\maj{\mathrm{maj}}
\def\syz{\mathrm{syz}}
\def\conv{\mathrm{conv}}
\def\Hilb{\mathrm{Hilb}}
\def\Tor{\mathrm{Tor}}
\def\init{\mathrm{init}}
\def\weak{\mathrm{weak}}
\def\strict{\mathrm{strict}}
\def\Kdim{\mathrm{dim}}

\def\gr{\mathfrak{gr}}

\begin{document}

\title[$P$-partitions revisited]
{$P$-partitions revisited}

\author{Valentin F\'eray}
\email{feray@labri.fr}
\address{LaBRI\\
Universit\'e Bordeaux 1\\
351 Cours de la Lib\'eration\\
33400 Talence\\
France}

\author{Victor Reiner}
\email{reiner@math.umn.edu}
\address{School of Mathematics\\
University of  Minnesota\\
Minneapolis, MN 55455\\
USA}


\thanks{Second author supported by NSF grant DMS-1001933.}



\begin{abstract}
We compare a traditional and non-traditional view on the
subject of $P$-partitions, leading to formulas counting
linear extensions of certain posets.
\end{abstract}

\maketitle


\section{Introduction}
\label{intro-section}

Our goal is to re-examine Stanley's theory of
$P$-partitions from a non-traditional viewpoint, one that
arose originally from ring-theoretic considerations
in \cite{BFLR}.  Comparing viewpoints, for example, gives
an application to counting linear extensions of certain posets.
We describe these viewpoints here, followed by this enumerative
application, and then give an indication of the further ring-theoretic results.

\subsection{Traditional viewpoint}
\label{sec:traditional-viewpoint}
Given a partial order $P$ on the set $\{1,2,\ldots,n\}$ 
a {\it weak $P$-partition} \cite[\S 4.5]{Stanley-EC} is a 
map $f: P \rightarrow \NN:=\{0,1,2,\ldots\}$
satisfying $f(i) \geq f(j)$ whenever $i <_P j$.

In Stanley's original work \cite{Stanley-thesis} and
that of A. Garsia \cite{Garsia}, 
it was important that one can express a weak $P$-partition $f$ uniquely as a sum 
$
f=\chi_{I_1}+\chi_{I_2}+\cdots+\chi_{I_{\max(f)}}
$
of indicator functions $\chi_{I_i}$ for a multiset of
nonempty, nested order ideals $I_i$ in $P$; specifically
$I_i:=\{j \in P: f(j) \geq i\}$.  An important
special case occurs when $f$ takes on each
value in $\{1,2,\ldots,n\}$ exactly once, so that
the nested sequence of order ideals 
$
I_1 \supset \cdots \supset I_n \supset I_{n+1}:=\varnothing
$
corresponds to a permutation $w=(w(1),\ldots,w(n))$ of $\{1,2,\ldots,n\}$ 
defined by $w(i)=I_i \setminus I_{i+1}$.  Such permutations $w$ are
called {\it linear extensions} of $P$ because the order $<_w$ given by 
$w(1) <_w \cdots <_w w(n)$ strengthens the partial order $P$
to a linear order.  

This has a geometric interpretation:
the weak $P$-partitions $f$ are the integer points inside a 
rational polyhedral cone in $\RR^n$ defined
by the inequalities $f_i \geq f_j \geq 0$ for $i <_P j$, 
and the set $\LLL(P)$ of all 
linear extensions of $P$ indexes the maximal simplicial
subcones in a unimodular triangulation of this $P$-partition cone.
The simplicial complex underlying this triangulation is
the order complex for the finite distributive lattice structure
on the set $\JJJ(P)$ of all order ideals in $P$;
see \cite{Stanley-poset-polytopes}.

\subsection{New viewpoint}
Here a much larger role is played by the subset $\Jconn(P) \subset \JJJ(P)$
consisting of all nonempty {\it connected} order ideals $J$ in $P$, that is,
those order ideals $J$ whose Hasse diagram is a connected
graph.  Say that two connected order ideals $J_1, J_2$ 
{\it intersect trivially} if either they are disjoint
or they are nested, that is, comparable under inclusion;  
otherwise say that they {\it intersect nontrivially}.  

It will be important
that one can express a $P$-partition $f$ uniquely as a sum 
\begin{equation}
\label{f-expressed-as-connected-ideals}
f=\chi_{J_1} + \chi_{J_2} + \cdots + \chi_{J_{\nu(f)}}
\end{equation}
of the indicator functions $\chi_{J_i}$ where $\{J_1,J_2,\ldots,J_{\nu(f)}\}$
is a multiset of nonempty connected order ideals in $P$ that
pairwise intersect trivially; specifically
one takes the $\{J_\ell\}_{\ell=1}^{\nu(f)}$ to be the multiset of
connected components of the Hasse diagrams for the
order ideals $I_i=\{j \in P: f(j) \geq i\}$ mentioned earlier.

Geometrically, this corresponds to a different (non-unimodular) 
triangulation of the $P$-partition cone.  This triangulation
is intimately related to the refinement of the normal
fan of a {\it graphic zonotope} by the normal fan of one
of Carr and Devadoss's {\it graph-associahedra} \cite{CarrDevadoss};
see Section~\ref{triangulation-section}.

\subsection{Counting linear extensions}
Computing the number $|\LLL(P)|$ of linear extensions of $P$
for general posets $P$ is known to be a $\# P$-hard problem
by work of Brightwell and Winkler \cite{BrightwellWinkler}.  However, for the
class of posets which we are about to define, a formula
for  $|\LLL(P)|$ will follow easily 
from the above considerations.

Say that a finite poset $P$ is a {\it forest with duplications}
if it can be constructed from one-element posets by iterating
the following three operations:

\vskip.1in
\noindent
{\sf Disjoint union:}
Given two posets $P_1, P_2$, form 
their {\it disjoint union} $P_1 \sqcup P_2$,
in which all elements of $P_1$ are incomparable
to all elements of $P_2$.

\vskip.1in
\noindent
{\sf Hanging:}
Given two posets $P_1, P_2$
and any element $a$ in $P_1$, form a new poset by {\it hanging} $P_2$ {\it below
$a$ in $P_1$}, that is, add to the disjoint union $P_1 \sqcup P_2$ all the order
relations $p_2 < b$ for every $p_2$ in $P_2$ and $b$ in $P_1$ with $b \geq_{P_1} a$.

\vskip.1in
\noindent
{\sf Duplication of a hanger:}
Say that an element $a$ in $P$ is a {\it hanger}
if $P$ can be formed by hanging the
nonempty subposet $P_2:=P_{<a}$ below $a$ in the subposet
$P_1:=P \setminus P_{<a}$.
Equivalently, $a$ is hanger in $P$ 
if $P_{<a}$ is nonempty and every path in the Hasse 
diagram of $P$ from an element of $P_{< a}$
to an element of $P \setminus P_{\leq a}$ must pass through $a$.
Then one can form the {\it duplication of the hanger $a$ in $P$} 
with duplicate element $a'$:
add to the disjoint union $P \cup \{a'\}$ all order relations
$p < a'$ (respectively $a'<p$) whenever $p <_P a$ (respectively
$a <_P p$).
\vskip.1in
\noindent 
Note that when one disallows the duplication-of-hanger operation
from the above list of constructions, one obtains the subclass of
{\it forest posets}, that is,
posets in which every element is covered by at most one other element.

\begin{figure}
\epsfxsize=100mm
\epsfbox{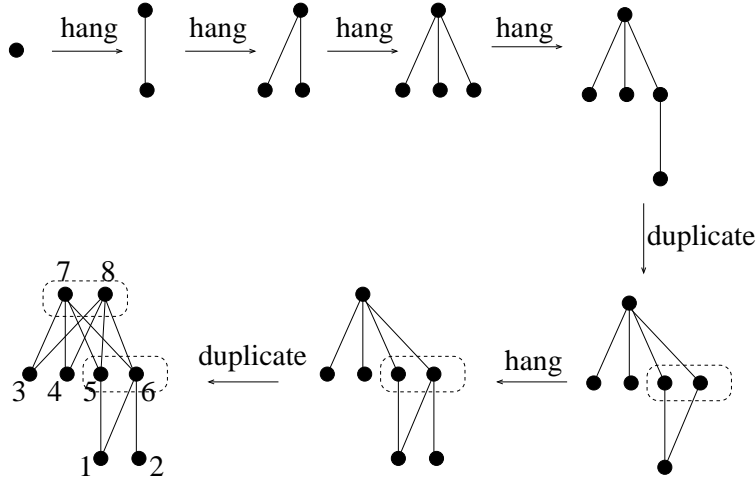}
\caption{A duplicated forest built by a sequence of hanging
and duplication operations.}
\label{duplicated-forest-figure}
\end{figure}

For the sake of stating our first main result counting linear
extensions, we define the notion of a naturally labelled poset $P$:
it means that $i <_P j$ implies $i <_\ZZ j$.
Let us also recall the {\it major index} statistic on a
permutation $w=(w(1),\ldots,w(n))$ defined by
$$
\maj(w):=\sum_{\substack{i=1,2,\ldots,n-1:\\w(i) > w(i+1)}} i
$$
and these standard $q$-analogues of the number $n$
and the factorial $n!$:
$$
\begin{aligned}[]
[n]_q &:=1+q+q^2+\cdots+q^{n-1} = \frac{1-q^n}{1-q}\\
[n]!_q &:=[1]_q [2]_q \cdots [n-1]_q [n]_q.\\
\end{aligned}
$$

We give a proof of the following result by inclusion-exclusion
in Section~\ref{first-forest-proof-section}, and 
then a second proof via commutative algebra in
Section~\ref{second-forest-proof-section}.

\begin{theorem}
\label{thm:duplicated-forest-extensions}
Let $P$ be a naturally labelled forest with duplications on
$\{1,2,\ldots,n\}$.  Then
\begin{equation}
\label{duplicated-forest-q-formula}
\sum_{w \in \LLL(P)} q^{\maj(w)}=
[n]!_q \cdot {\prod_{\{J_1,J_2\} \in \Pi(P)} 
         \left[ \,\,\, |J_1|+|J_2| \,\,\, \right]_q } \Bigg /
               {\prod_{J \in \Jconn(P)} \left[ \,\,\, |J| \,\,\, \right]_q }
\end{equation}
where the product in the numerator runs over all
the set $\Pi(P)$ consisting of all pairs $\{J_1,J_2\}$ 
of connected order ideals of $P$ that intersect nontrivially.
In particular, upon setting $q=1$, one has
\begin{equation}
\label{duplicated-forest-formula}
| \LLL(P) | =
n! \cdot {\prod_{\{J_1,J_2\} \in \Pi(P)}  (|J_1|+|J_2|) } \Bigg /
               {\prod_{J \in \Jconn(P)} |J| }.
\end{equation}
\end{theorem}

\noindent
The products appearing in Theorem~\ref{thm:duplicated-forest-extensions} are
much more explicit than they first appear, as it will
be shown (see Lemma~\ref{lemma:duplicated-forest-structure}) 
that for a forest $P$ with duplications, the two sets
$\Jconn(P)$ and $\Pi(P)$ are easily written down in terms of
the {\it principal ideals} $P_{\leq p}$ and 
the {\it duplication set} $\DDD(P)$ consisting of all duplication 
pairs $\{a,a'\}$ that were created during the various 
steps that build $P$:
\begin{equation}
\label{forest-with-duplication-ideal-description}
\begin{aligned}
\Jconn(P) &= \{ P_{\leq p}\}_{ p \in P } \quad \sqcup \quad
            \{ P_{\leq a,a'} \}_{\{ a,a'\} \in \DDD(P) } \\
\Pi(P) &= \{ \,\,\, \{ P_{\leq a}, P_{\leq a'} \}\,\,\,  \}_{\{ a,a'\} \in \DDD(P)}. \\
\end{aligned}
\end{equation}

\noindent
Figure~\ref{duplicated-forest-figure} 
shows an example of a {\it forest with duplications} $P$
built by a sequence of hangings and duplications; 
no disjoint union operations are used, yielding only one
connected component.  Its duplication set 
$
\DDD(P)=\{ \{5,6\}, \{7,8\} \}
$
is shown dotted.  One has the following list of cardinalities $|J|$ of connected order ideals $J$
\vskip.1in
\noindent
$$
\begin{aligned}
&\begin{tabular}{|c||c|c|c|c|c|c|c|c|c|c|}\hline
 & & & & & & & &  \\
$J \in \Jconn(P)$ & 
$P_{\leq 1}$ & 
$P_{\leq 2}$ & 
$P_{\leq 3}$ & 
$P_{\leq 4}$ & 
$P_{\leq 5}$ &
$P_{\leq 6}$ &
$P_{\leq 7}$ &
$P_{\leq 8}$ \\
 & & & & & & & &  \\\hline
$|J|$ & $1$ & $1$ & $1$ & $1$  & $2$ & $3$ & $7$ & $7$ \\\hline
\end{tabular} \\
&\begin{tabular}{|c||c|c|}\hline
 & & \\
$J \in \Jconn(P)$ & $P_{\leq 5} \cup P_{\leq 6}$ & 
$P_{\leq 7} \cup P_{\leq 8}$ \\
 & & \\ \hline
$|J|$ & $4$ & $8$\\\hline
\end{tabular}
\end{aligned}
$$
and this data on the pairs in $\Pi(P)$
$$
\begin{tabular}{|c||c|c|}\hline
 & & \\
$\{J_1,J_2\} \in \Pi(P)$ 
& $\{ P_{\leq 5}, P_{\leq 6} \}$  
& $\{ P_{\leq 7}, P_{\leq 8} \}$ \\
 & & \\ \hline
$|J_1|+|J_2|$ & $2+3=5$ & $7+7=14$ \\\hline
\end{tabular}
$$

\vskip.1in
\noindent
Consequently, Theorem~\ref{thm:duplicated-forest-extensions}
implies that 
$$
\begin{aligned}
\sum_{w \in \LLL(P)} q^{\maj(w)} 
&=\frac{ [8]!_q }
    {[1]_q \cdot [1]_q \cdot [1]_q \cdot [1]_q \cdot [2]_q \cdot [3]_q \cdot [7]_q \cdot [7]_q }  
      \cdot \frac{ [5]_q \cdot [14]_q }{ [4]_q \cdot [8]_q }\\
&=\frac{[5]_q \cdot [5]_q \cdot [6]_q \cdot [14]_q}
          {[7]_q} 
= [2]_{q^7} \cdot [5]_{q} \cdot [5]_{q} \cdot [6]_q  \\
   \end{aligned}
$$
and upon setting $q=1$, one obtains
$$
|\LLL(P)|
= 2 \cdot 5 \cdot 5 \cdot 6 = 300.
$$
This example has been checked using the software {\emph sage} \cite{sage}, see 

{\tt http://www.sagenb.org/home/pub/2701/}.

A special case of Theorem~\ref{thm:duplicated-forest-extensions}
is well-known, namely when the forest
has {\it no duplications}, and the set $\Pi(P)$ is empty.
In this case, one simply has a {\it forest poset}.
Then equation \eqref{duplicated-forest-formula}
becomes Knuth's well-known {\it hook formula for 
linear extensions of forests} \cite[\S 5.1.4 Exer. 20]{Knuth},
and  equation \eqref{duplicated-forest-q-formula}
becomes Bjorner and Wachs' more general
{\it major index $q$-hook formula for forests}
\cite[Theorem 1.2]{BjornerWachs}.
The derivation of these two special cases
from consideration of $P$-partition rings was already pointed 
out in \cite[\S 6]{BFLR};  see also
Examples~\ref{maj-forest-formula-example} and
\ref{maj-forest-with-duplication-formula-example} below.

\subsection{The ring of weak $P$-partitions}

Although Theorem~~\ref{thm:duplicated-forest-extensions}
has a simple combinatorial proof, it was not our
original one.  We were motivated from
trying to understand the structure of the {\it affine semigroup ring}
$R_P$ of $P$-partitions,the 
subalgebra of the polynomial ring $k[x_1,\ldots,x_n]$
spanned $k$-linearly by the monomials
$\xx^f:=x_1^{f(1)} \cdots x_n^{f(n)}$ as $f$ runs through all
weak $P$-partitions.  In \cite{BFLR} this was the ring denoted $R^{\mathrm{wt}}_P$.
There it was noted that a minimal generating set as
an algebra is given by the monomials $\xx^J:=\prod_{j \in J} x_j$
as $J$ runs through the set $\Jconn(P)$ of nonempty connected order ideals
of $P$.  We extend this to the following result in
Section~\ref{presentations-section}.

\begin{theorem}
\label{thm:minimal-presentation}
For any poset on $\{1,2,\ldots,n\}$ and any field $k$,
the $P$-partition ring $R_P$ has minimal presentation 
$$
0 \rightarrow I_P \longrightarrow  S
 \overset{\varphi}{\longrightarrow} R_P \rightarrow 0
$$
in which the polynomial algebra 
$S=k[U_{J}]_{J \in \Jconn(P)}$ maps to $R_P$ via
$U_J \overset{\varphi}{\longmapsto} \xx^J$, and the kernel
ideal $I_P$ has a minimal generating set
indexed by $\{J_1,J_2\}$ in $\Pi(P)$, consisting
of binomials 
\begin{equation}
\label{P-partition-syzygy}
\syz_{J_1,J_2}:=U_{J_1} U_{J_2} - 
  U_{J_1 \cup J_2} \cdot U_{J^{(1)}}  U_{J^{(2)}} \cdots  U_{J^{(t)}}
\end{equation}
where the intersection $J_1 \cap J_2$ has connected component 
ideals $J^{(1)} \sqcup \cdots \sqcup J^{(t)}$.
\end{theorem}

\vskip.1in
\noindent
{\bf Example.}
For the poset in Figure~\ref{duplicated-forest-figure}, 
the presentation of Theorem~\ref{thm:minimal-presentation} is $R_P=S/I_P$, where
\begin{equation}
\label{example-polynomial-algebra}
S=k\left[U_1,
    U_2,
    U_3,
    U_4,
    U_{15},
    U_{126},
    U_{1234567},
    U_{1234568},
    U_{1256},
    U_{12345678}\right]
\end{equation}
and $I_P$ is the ideal of $S$ generated by
$$
\begin{aligned}
U_{15} U_{126} &- U_{1256} U_{1}, \\
U_{1234567} U_{1234568} &- U_{12345678} U_{1256} U_{3} U_{4}
\end{aligned}
$$

It is not hard to see 
(and explained in Corollary~\ref{associated-graded-Hilbert-series-corollary})
how the various generating functions for
(weak) $P$-partitions turn into Hilbert series calculations
for $R_P$.  This suggests trying to understand the structure
of $R_P$ in order to calculate its Hilbert series.
One natural situation where this follows easily is when
$R_P \cong S/I_P$ gives a {\it complete intersection presentation},
that is, the Krull dimension $n$ of $R_P$ plus the 
size $|\Pi(P)|$ of the minimal generating set for $I_P$ 
sums to the Krull dimension $|\Jconn(P)|$ of $S$.
The forward implication in the following combinatorial
characterization of the complete intersection case
is proven in Section~\ref{second-forest-proof-section}, and used
to give our second (but historically first) 
proof of Theorem~~\ref{thm:duplicated-forest-extensions}:

\begin{theorem}
\label{thm:c.i.-characterization}
A poset $P$ on $\{1,2,\ldots,n\}$ is a forest with duplications
if and only if $R_P=S/I_P$ is a complete intersection presentation.
\end{theorem}

\subsection{The associated graded ring}
We explain in Section~\ref{associated-graded-section} the significance of the
statistic $\nu(f)$ on a $P$-partition $f$ which
appeared in the unique expression \eqref{f-expressed-as-connected-ideals}
above.  It turns out that $\nu(f)$ gives the $\NN$-grading of the image
of the monomial $\xx^f$
in the {\it associated graded ring} $\gr(R_P) =\gr_\mm(R_P)$ 
with respect to the unique $\NN$-homogeneous 
maximal ideal $\mm \subset R_P$.  Consequently, $\gr(R_P)$
has $\NN \times \NN^n$-graded Hilbert series 
\begin{equation}
\label{finest-graded-Hilbert-series}
\Hilb(\gr(R_P),t,\xx) = \sum_{f \in \AAA^\weak(P)} t^{\nu(f)} \xx^f.
\end{equation}
An expression for this Hilbert series as a summation over
the set $\LLL(P)$ of linear extensions of $P$ is given in
\eqref{generating-function-in-t,x} below\footnote{Assuming that
$P$ has been {\it naturally labelled}; 
see Remark~\ref{natural-labelling-remark} below.}.
The following presentation and initial ideal for  $\gr(R_P)$
will be derived in Section~\ref{presentations-section}.

\begin{theorem}
\label{thm:two-other-ideal-generators}
For any poset on $\{1,2,\ldots,n\}$ and any field $k$,
the associated graded ring $\gr(R_P)$ has minimal presentation 
$0 \rightarrow I^\gr_P \longrightarrow  S
 \overset{\gr(\varphi)}{\longrightarrow} \gr(R_P) \rightarrow 0
$
in which the polynomial algebra 
$S=k[U_{J}]_{J \in \Jconn(P)}$ is mapped to $\gr(R_P)$ via
$U_J \overset{\varphi}{\longmapsto} \overline{\xx}^J$, and the kernel
ideal $I^\gr_P$ has minimal generating
indexed by $\{J_1,J_2\}$ in $\Pi(P)$, consisting
of the quadratic binomials and monomials
\begin{equation}
\label{gr-syzygy}
\syz^\gr_{J_1,J_2}:=
\begin{cases} 
U_{J_1} U_{J_2} - U_{J_1 \cup J_2} U_{J_1 \cap J_2}
  & \text{ if } J_1 \cap J_2 \text{ is connected},\\
U_{J_1} U_{J_2} & \text{ if } J_1 \cap J_2 \text{ is disconnected},\\
\end{cases}.
\end{equation}
Furthermore, there exist monomial orders on $S$ for which
the initial ideal of both $I_P$ and $I^\gr_P$ is the squarefree
quadratic monomial ideal $I^\init_P$ having minimal generating set
indexed by $\{J_1,J_2\}$ in $\Pi(P)$, consisting
of the squarefree quadratic monomials
\begin{equation}
\label{monomial-syzygy}
\syz^\init_{J_1,J_2}:=U_{J_1} U_{J_2}.
\end{equation} 
\end{theorem}

\vskip.1in
\noindent
{\bf Example.}
For the poset in Figure~\ref{duplicated-forest-figure}, 
the presentation in Theorem~\ref{thm:two-other-ideal-generators} is 
$\gr(R_P)=S/I^\gr_P$, where
$S$ is as in \eqref{example-polynomial-algebra} and 
$I^\gr_P$ is generated by the binomial 
$$
\begin{aligned}
&U_{15} U_{126} - U_{1256} U_{1}, \\
&U_{1234567} U_{1234568}
\end{aligned}
$$
while the initial ideal $I^\init_P$ in 
Theorem~\ref{thm:two-other-ideal-generators}
is generated by the monomials
$$
\begin{aligned}
&U_{15} U_{126}, \\
&U_{1234567} U_{1234568}.
\end{aligned}
$$

The existence of this quadratic initial ideal $I^\init_p$
has this consequence.
\begin{corollary}
\label{Koszul-Hilb-corollary}
For any poset $P$ on $\{1,2,\dots,n\}$, the associated
graded ring $\gr(R_P)$ a Koszul algebra.  Thus the 
$\NN \times \NN^n$-graded Hilbert series
from \eqref{finest-graded-Hilbert-series} will have the property that 
$$
\left[ \sum_{f \in \AAA^\weak(P)} t^{\nu(f)} \xx^f \right]^{-1}_{t \mapsto -t}
$$
is a power series in $t,x_1,\ldots,x_n$ with nonnegative coefficients;
specifically, it is the $\NN \times \NN^n$-graded Hilbert series 
of the Koszul dual algebra $\gr(R_P)^!$.
\end{corollary}

The remainder of the paper explains these results further.
The reader interested only in the combinatorial results 
will find them in Sections \ref{unique-expressions-section} through
\ref{first-forest-proof-section}, and can safely skip the connections to
ring-theory explained in Sections~\ref{associated-graded-section} through 
\ref{CI-characterization-section}.
Section~\ref{triangulation-section} discusses the geometry of
the initial ideal $I^\init_P$ and its associated triangulation of
the $P$-partition cone, relating it to graphic zonotopes and
graph-associahedra. Section~\ref{questions-section}
collects some further questions.

\section{Unique expressions}
\label{unique-expressions-section}

We discuss some old and new ways to uniquely express a $P$-partition,
mentioned in the Introduction.  

\begin{definition}
Let $P$ be a  partial order $<_P$ on $\{1,2,\ldots,n\}$,
and consider the nonnegative integers $\NN=\{0,1,2,\ldots\}$ as a totally
ordered set with its usual order $<_\NN$.
Say that a map $f: P \rightarrow \NN:=\{0,1,2,\ldots\}$ is
\begin{enumerate}
\item[$\bullet$] 
a {\it weak $P$-partition} if it is weakly order-reversing:
$i <_P j$ implies the inequality $f(i) \geq_\NN f(j)$;
\item[$\bullet$] 
a {\it $P$-partition} if, in addition,
whenever $i <_P j$ and $i >_\NN j$, one has a strict inequality $f(i) >_\NN f(j)$;
\item[$\bullet$] 
a {\it strict $P$-partition} if 
$i <_P j$ implies $f(i) >_\NN f(j)$.
\end{enumerate}
{\sf NB:} This terminology is similar in spirit, but not quite 
the same as those used by Stanley in \cite[\S 4.5, \S 7.19]{Stanley-EC}.  
We hope that the slight differences create no confusion.
\end{definition}
\noindent
Denoting by $\AAA(P), \AAA^{\weak}(P),\AAA^{\strict}(P)$ the sets
of $P$-partitions, weak $P$-partitions, and strict $P$-partitions,
one has the inclusions
\begin{equation}
\label{P-partition-inclusions}
\AAA^{\strict}(P) \subseteq \AAA(P) \subseteq \AAA^{\weak}(P).
\end{equation}
One has equality in the second inclusion of \eqref{P-partition-inclusions}
if and only if $P$ is {\it naturally labelled};
similarly one has equality in the first inclusion
of \eqref{P-partition-inclusions} if and only if 
$P$ is {\it strictly} or {\it anti-naturally labelled} 
in the sense that $i <_P j$ implies $i >_\NN j$.

\begin{example}
\label{ex:three-little-posets}
The three posets $P_1, P_2, P_3$ on $\{1,2,3\}$
shown below 

\epsfxsize=60mm
\centerline{\epsfbox{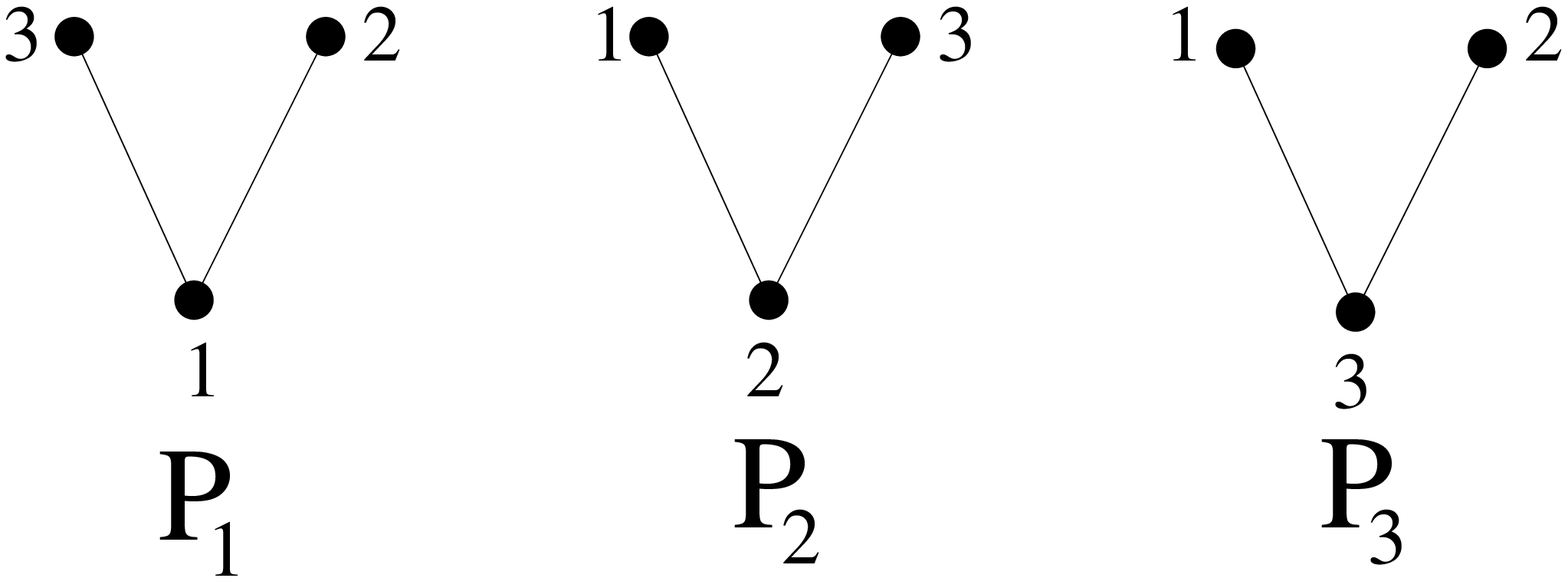}}
\noindent
 are all isomorphic, with $P_1$ naturally labelled,
$P_3$ strictly labelled, and $P_2$ neither naturally
nor strictly labelled.  One has
$$
\begin{aligned}
\AAA(P_1)&=\{ f=(f(1),f(2),f(3)) \in \NN^3: f(1) \geq_\NN f(2),f(3) \} \\
\AAA(P_2)&=\{ f=(f(1),f(2),f(3)) \in \NN^3: f(2) \geq_\NN f(3) \text{ and } f(2) >_\NN f(1) \} \\
\AAA(P_3)&=\{ f=(f(1),f(2),f(3)) \in \NN^3: f(3) >_\NN f(1), f(2) \}.
\end{aligned}
$$
\end{example}

\begin{definition}
Recall that a permutation $w=(w(1),\ldots,w(n))$ of $\{1,2,\ldots,n\}$ 
is a {\it linear extension} of $P$ if the order $w(1) <_w \cdots <_w w(n)$
extends $P$ to a linear order.  Denote by $\LLL(P)$ the set of
all linear extensions $w$ of $P$. 
Denote by $w|_{[1,i]}$ the
initial segment $\{w(1),w(2),\ldots,w(i)\}$ of $w$ thought of as a subset
of $\{1,2,\ldots,n\}$.
It is an order ideal of $P$ whenever $w$ lies in $\LLL(P)$.

For any subset $A \subset \{1,2,\ldots,n\}$, let
$\chi_A$ be its characteristic function thought of as a vector
in $\{0,1\}^n$.
\end{definition}

\begin{proposition}
\label{fundamental-P-partition-proposition}
For any poset $P$ on $\{1,2,\ldots,n\}$, and
any $P$-partition $f$, there
exists a unique permutation $w$ in $\LLL(P)$ for which
\begin{equation}
\label{weak-part-of-Pw-partition}
f(w(1)) \geq \cdots \geq f(w(n))
\end{equation}
and one has strict inequality $f(w(i)) > f(w(i+1))$ when $w(i) > w(i+1)$,
that is, whenever $i$ is an element of the descent set $\Des(w)$.
Consequently, 
$$
f = \sum_{i=1}^n (f(w(i))-f(w(i+1))) \cdot \chi_{w|_{[1,i]}}.
$$
\end{proposition}
\begin{proof}
(See \cite[Lemma 4.5.1, Theorem 7.19.4]{Stanley-EC})
One takes $w$ to 
be the minimum length or lexicographically earliest
permutation satisfying \eqref{weak-part-of-Pw-partition}.
\end{proof}

\begin{proposition}
\label{P-partition-expressions-prop}
For any poset $P$ on $\{1,2,\ldots,n\}$, any weak $P$-partition $f$
(and hence also any $P$-partition) has a unique expression as
\begin{enumerate}
\item[(i)] $f = \sum_{i=k}^{\max(f)} \chi_{I_k}$
for a multiset $I_1 \supseteq \cdots \supseteq I_{\max(f)}$ of 
nested nonempty order ideals in $P$, and also as
\item[(ii)] $f = \sum_{i=1}^{\nu(f)} \chi_{J_i}$
for a multiset $J_1,J_2,\ldots,J_{\nu(f)}$ of nonempty connected order 
ideals of $P$ which pairwise intersect trivially.
\end{enumerate}
\end{proposition}
\begin{proof}
For (i), one sets $I_k:=f^{-1}(\{k,k+1,k+2,\ldots\})$ for
$k=1,2,\ldots,t:=\max(f)$.  

To prove (ii), one can show existence of such an expression for $f$
by starting with the multichain $J_1 \supseteq \cdots \supseteq J_t$
of order ideals from (i), and replacing each order ideal $J_i$ with
its collection of connected components.  It is not hard to see
that the resulting multiset of connected order ideals will pairwise
intersect trivially.  

To prove uniqueness of the expression in (ii),
induct on $|f|:=\sum_{i=1}^n f(i)$, with trivial base case $f=0$.
In the inductive step, let $f \neq 0$, and consider the set $J$ which is
the support of $f$ as a subset of $P$.  Because $f$ is a $P$-partition, 
$J$ is a nonempty order ideal. Decompose $J$ into its connected components 
$J^{(1)},J^{(2)},\ldots,J^{(c)}$, which are all connected order ideals.  

If $c > 1$, then one can consider for each $i$ the restriction 
$f|_{J^{(i)}}$ as a $J^{(i)}$-partition.  Since $|J^{(i)}|<|J| \leq |P|$,
uniqueness follows by induction.  

If $c=1$, so that $J$ is connected (and nonempty), 
then $f=\chi_J + \hat{f}$, where $\hat{f}$ is again a $P$-partition,
and $|\hat{f}| < |f|$.  Again, uniqueness follows by induction.
\end{proof}

\begin{remark}
\label{natural-labelling-remark}
The relation between Propositions~\ref{fundamental-P-partition-proposition}
and \ref{P-partition-expressions-prop}
is easiest when $P$ is naturally labelled, so that 
a $P$-partition $f$ is the same as a weak $P$-partition.
In that case, the unique permutation $w$ guaranteed by 
Proposition~\ref{fundamental-P-partition-proposition}
has the property that the multiset of ideals $\{ I_k \}_{k=1,2,\ldots,\max(f)}$
contains the order ideal $w|_{[1,i]}$ of $P$
with multiplicity $f(w(i))-f(w(i+1))$.

We also note that it is essentially innocuous to relabel an arbitrary
poset $P$ so that it is naturally labelled, either if the goal is 
to count the linear extensions $\LLL(P)$, or if the goal is
to understand the ring $R_P$-- this ring depends
only on $P$ up to isomorphism, not on the labelling.  The labelling
of $P$ only makes a difference when considering the ideal $\III_P$ within $R_P$
consider later, in Section~\ref{non-natural-labelings-section}.
\end{remark}

\begin{example}
Let $P$ be the naturally labelled poset on $\{1,2,3,4,5,6,7,8,9\}$
from Figure~\ref{duplicated-forest-figure}
and let $f$ be the $P$-partition with values in the following table, as depicted below:
$$
\begin{tabular}{|c||c|c|c|c|c|c|c|c|c|}\hline
$i$    & $1$ & $2$ & $3$ & $4$ & $5$ & $6$ & $7$ & $8$ \\\hline
$f(i)$ & $5$ & $4$ & $2$ & $1$ & $2$ & $4$ & $0$ & $1$ \\\hline
\end{tabular}
$$
$$
\epsfxsize=70mm
\epsfbox{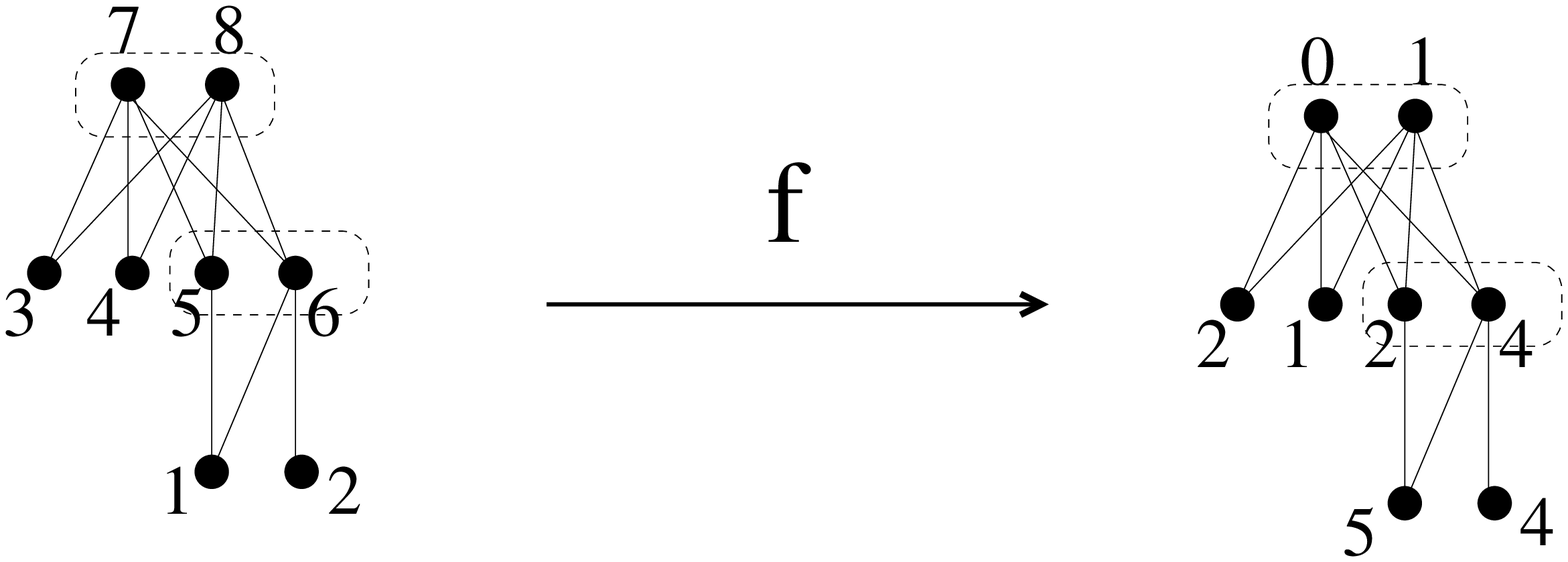}
$$
Then $\max(f)=5$ and the unique expression for $f$ as in 
Proposition~\ref{P-partition-expressions-prop} part (i)
is $f=\sum_{j=1}^5\chi_{I_j}$ where $\{I_1,I_2,I_3,I_4,I_5\}$
are the nested ideals shown here
$$
\begin{array}{cccccccccc}
I_1   &\supseteq &I_2 &\supseteq &I_3 &=& I_4& \supseteq& I_5 \\
\Vert   &        &\Vert &        &\Vert &\      & \Vert&        & \Vert \\
\{1,2,3,4,5,6,8\} & &\{1,2,3,5,6\}& &\{1,2,6\}& &\{1,2,6\}& &\{1\}\\
\Vert   &        &\Vert &        &\Vert &\      & \Vert&        & \Vert \\
J_1 &   &\{1,2,5,6\} \sqcup \{3\}&   &J_3 &   & J_4&    & J_5 \\
    &          &\Vert&         &    &        &    &        &  \\
    &          &J_2 \sqcup J_6&         &    &        &    &        &  \\
\end{array}
$$
and these decompose into the multiset of $\nu(f)=6$ 
connected component ideals 
$
\{J_1,J_2,J_3,J_4,J_5,J_6\}
$
labelled above, giving the expression $f=\sum_{i=1}^6\chi_{J_i}$ as in
Proposition~\ref{P-partition-expressions-prop} part (ii).
The unique expression as in 
Proposition~\ref{fundamental-P-partition-proposition}
has
$$
\begin{array}{rlccccccl}
w&=(w(1),&w(2),&w(3),&w(4),&w(5),&w(6),&w(7),&w(8)) \\
&=(1,&2,&6,&3,&5,&4,&8,&7)
\end{array}
$$
and
$
f=
 1 \cdot \chi_{w|_{[1,1]}}
+2 \cdot \chi_{w|_{[1,3]}}
+1 \cdot \chi_{w|_{[1,5]}}
+1 \cdot \chi_{w|_{[1,7]}}.
$

\end{example}

\section{Generating functions}
\label{generating-functions-section}

We explain here how Proposition~\ref{P-partition-expressions-prop} 
suggests generating functions counting $P$-partitions
and linear extensions according to certain statistics, which
one can then specialize in various ways.  
We will see in Corollary~\ref{associated-graded-Hilbert-series-corollary} 
that they all have natural interpretations as 
Hilbert series for the $P$-partition
ring $R_P$ or its associated graded ring $\gr(R_P)$
using different specializations of their multigradings.

\begin{definition}
Given a $P$-partition $f$, recall that $\nu(f)$ denotes
the size (counting multiplicity) of the multiset 
$\{J_1,\ldots,J_{\nu(f)}\}$ in the unique expression
\eqref{f-expressed-as-connected-ideals} whose existence is guaranteed by 
Proposition~\ref{P-partition-expressions-prop}(ii).  

Given an order ideal $J$ of $P$, let $c_P(J)$ denote
the number of connected components in the Hasse diagram of the
restriction $P|_J$.  We also define a 
new descent statistic for $w$ that {\it depends upon the poset 
structure of $P$}:
$$
\des_P(w) := \sum_{ i \in \Des(w) } c_P(w|_{[1,i]}) 
$$ 
\noindent
Recall also that we have been using the notations
$
\xx^f:=x_1^{f(1)} \cdots x_n^{f(n)}
$ 
for $f \in \NN^n$, and
$
\xx^A:=\prod_{i \in A} x_i
$ 
for subsets $A \subseteq \{1,2,\ldots,n\}.$.
\end{definition}

\begin{corollary}
\label{generating-functions-corollary}
For any poset $P$ on $\{1,2,\ldots,n\}$, one has
\begin{equation}
\label{generating-function-in-t,x}
\sum_{f \in \AAA(P)} t^{\nu(f)} \xx^f =
\sum_{w \in \LLL(P)} \frac{ t^{\des_P(w)} 
                        \prod_{i \in \Des(w)}\xx^{w|_{[1,i]}} }
                  { \prod_{i=1}^n (1 - t^{c_P(w|_{[1,i]})} \xx^{w|_{[1,i]}} ) }.
\end{equation}

Setting $t=1$ in \eqref{generating-function-in-t,x} gives
\begin{equation}
\label{generating-function-in-x}
\sum_{f \in \AAA(P)} \xx^f =
\sum_{w \in \LLL(P)} \frac{ \prod_{i \in \Des(w)}\xx^{w|_{[1,i]}} }
                  { \prod_{i=1}^n (1 - \xx^{w|_{[1,i]}} ) },
\end{equation}
whereas setting $x_i=q$ for all $i$ in \eqref{generating-function-in-t,x} gives
\begin{equation}
\label{generating-function-in-t,q}
\sum_{f \in \AAA(P)} t^{\nu(f)} q^{|f|} =
\sum_{w \in \LLL(P)} \frac{ t^{\des_P(w)} q^{\maj(w)} }
                  { \prod_{i=1}^n (1 - t^{c_P(w|_{[1,i]})} q^i ) }.
\end{equation}
Further specializing $q=1$ in \eqref{generating-function-in-t,q} gives
\begin{equation}
\label{generating-function-in-t}
\sum_{f \in \AAA(P)} t^{\nu(f)} =
\sum_{w \in \LLL(P)} \frac{ t^{\des_P(w)} } { \prod_{i=1}^n (1 - t^{c_P(w|_{[1,i]})} ) }.
\end{equation}
Setting both $t=1$ and $x_i=q$ for all $i$ in \eqref{generating-function-in-t,x}
gives
\begin{equation}
\label{generating-function-in-q}
(1-q)(1-q^2)\cdots(1-q^n) 
  \sum_{f \in \AAA(P)} q^{|f|} = \sum_{w \in \LLL(P)} q^{\maj(w)},
\end{equation}
and hence, lastly,
\begin{equation}
\label{linear-extension-count-as-limit}
\lim_{q \rightarrow 1} (1-q)(1-q^2)\cdots(1-q^n) 
   \sum_{f \in \AAA(P)} q^{|f|} = |\LLL(P)|.
\end{equation}
\end{corollary}
\begin{proof}
To prove \eqref{generating-function-in-t,x},
use Proposition~\ref{P-partition-expressions-prop}(i) to write
the sum on the left as a sum over $w$ in $\LLL(P)$, and for each 
$P$-partition $f$, think about
how many connected order ideals (counted with multiplicity) will be
in the corresponding multiset from Proposition~\ref{P-partition-expressions-prop}(ii).
\end{proof}

We remark that the specializations to $t=1$ that appear 
in Corollary~\ref{generating-functions-corollary}, 
namely \eqref{generating-function-in-x} and its specializations
\eqref{generating-function-in-q}, \eqref{linear-extension-count-as-limit},
are all part of Stanley's traditional $P$-partition theory; see
\cite[\S 4.5]{Stanley-EC}.

\begin{example}
\label{ex:five-element-example}
For this naturally labelled poset $P$ on $\{1,2,3,4,5\}$
$$
\epsfxsize=20mm
\epsfbox{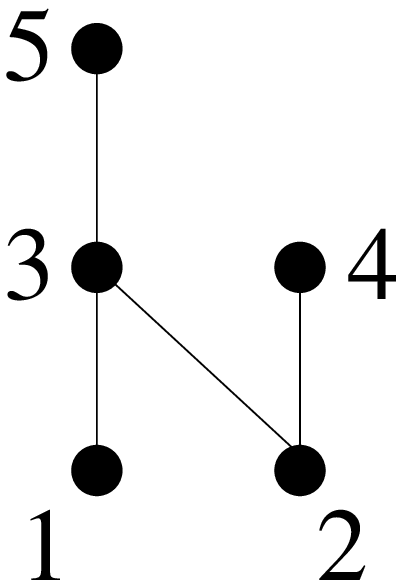}
$$
the expression in \eqref{generating-function-in-t,x} can be
computed using the following data

\begin{tabular}{|l|c|}\hline
\text{ nonempty ideal }$J \in \JJJ(P) $& $c_P(J)$ \\\hline\hline
$\{1\}$ & $1$ \\ \hline
$\{2\}$ & $1$ \\ \hline
$\{1,2\}$ & $2$ \\ \hline
$\{2,4\}$ & $1$ \\ \hline
$\{1,2,3\}$ & $1$ \\ \hline
$\{1,2,4\}$ & $2$ \\ \hline
$\{1,2,3,4\}$ & $1$ \\ \hline
$\{1,2,3,5\}$ & $1$ \\ \hline
$\{1,2,3,4,5\}$ & $1$ \\ \hline
\end{tabular}
\qquad
\begin{tabular}{|l|c|c|}\hline
$w \in \LLL(P)$ & $\des_P(w)$ \\\hline\hline
$12345$ & $0$ \\\hline
$1235\cdot 4$ & $1$ \\\hline
$124\cdot 35$ & $2$ \\\hline
$2\cdot1345$ & $1$ \\\hline
$2\cdot135\cdot4$ & $1+1=2$ \\\hline
$2\cdot14\cdot35$ & $1+2=3$ \\\hline
$24\cdot135$ & $1$ \\\hline
\end{tabular}
\vskip.1in
as the sum
$$
\begin{aligned}
&\sum_{f \in \AAA(P)} t^{\nu(f)} \xx^f =
\sum_{w \in \LLL(P)} \frac{ t^{\des_P(w)} 
                        \prod_{i \in \Des(w)}\xx^{w|_{[1,i]}} }
                  { \prod_{i=1}^n (1 - t^{c_P(w|_{[1,i]})} \xx^{w|_{[1,i]}} ) }=\\
&\frac{1}{(1-tx_1 )(1-t^2x_1 x_2 )(1-tx_1 x_2 x_3 )
    (1-tx_1 x_2 x_3 x_4 )(1-tx_1 x_2 x_3 x_4 x_5 )} +\\
&\frac{tx_1 x_2 x_3 x_5}{(1-tx_1 )(1-t^2x_1 x_2 )(1-tx_1 x_2 x_3 )
    (1-tx_1 x_2 x_3 x_5 )(1-tx_1 x_2 x_3 x_4 x_5 )} +\\
&\frac{t^2x_1 x_2 x_4}{(1-tx_1 )(1-t^2x_1 x_2 )(1-t^2x_1 x_2 x_4 )
    (1-tx_1 x_2 x_3 x_4 )(1-tx_1 x_2 x_3 x_4 x_5 )} +\\
&\frac{tx_2}{(1-tx_2 )(1-t^2x_1 x_2 )(1-tx_1 x_2 x_3 )
    (1-tx_1 x_2 x_3 x_4 )(1-tx_1 x_2 x_3 x_4 x_5 )} +\\
&\frac{tx_2 \cdot tx_1 x_2 x_3 x_5}{(1-tx_2 )(1-t^2x_1 x_2 )(1-tx_1 x_2 x_3 )
    (1-tx_1 x_2 x_3 x_5 )(1-tx_1 x_2 x_3 x_4 x_5 )} +\\
&\frac{tx_2 \cdot t^2x_1 x_2 x_4}{(1-tx_2 )(1-t^2x_1 x_2 )(1-t^2x_1 x_2 x_4 )
    (1-tx_1 x_2 x_3 x_4 )(1-tx_1 x_2 x_3 x_4 x_5 )} +\\
&\frac{tx_2 x_4}{(1-tx_2 )(1-tx_2 x_4 )(1-t^2x_1 x_2 x_4 )
    (1-tx_1 x_2 x_3 x_4 )(1-tx_1 x_2 x_3 x_4 x_5 )}
\end{aligned}
$$
which simplifies over a common denominator, after cancellations, to give
$$
\frac{1-t^2 \left(\xx^{(1,2,1,1,0)} + \xx^{(1,2,1,1,1)} + \xx^{(2,2,2,1,1)}\right)
+ t^3 \left(\xx^{(2,3,2,1,1)} + \xx^{(2,3,2,2,1)}\right)}
{\prod_{J \in \Jconn(P)} \left( 1-t \xx^J \right) }.
$$
The form of this last expression should be compared with
Corollary~\ref{associated-graded-Hilbert-series-corollary}(ii).
%
\end{example}

\section{First proof of Theorem~\ref{thm:duplicated-forest-extensions}:  inclusion-exclusion}
\label{first-forest-proof-section}

We begin the proof with the following lemma, partly
asserted already in the Introduction
as \eqref{forest-with-duplication-ideal-description}.  
Recall that for a forest with duplications $P$,
we denote by $\DDD(P)$ the collection of all pairs $\{a,a'\}$ that
arise by the duplication steps in the construction of $P$.
The set $\DDD(P)$ is well-defined
(it does not depend on the construction of $P$),
as shown by the following lemma.

\begin{lemma}
\label{lemma:duplicated-forest-structure}
   Let $P$ be a forest with duplications on $\{1,2,\ldots,n\}$. 
    \begin{itemize}
        \item[(i)] The duplication pairs in $\DDD(P)$
            are pairwise disjoint:  for any
            $\{a,a'\}, \{b,b'\}$ in $\DDD(P)$, 
            either $\{a,a'\}= \{b,b'\}$
            or $\{a,a'\} \cap \{b,b'\}=\varnothing$.
        \item[(ii)] The set $\Jconn(P)$ of nonempty connected order-ideals of $P$ are the principal
            ideals $P_{\leq p}$ (for $p \in P$), and the 
            unions $P_{\leq a} \cup P_{\leq a'}$ for $\{a,a'\}$ in $\DDD(P)$.
        \item[(iii)] The set $\Pi(P)$ of pairs $\{J_1,J_2\}$ of connected order-ideals 
            of $P$ intersecting non-trivially are the pairs 
            $\big\{P_{\leq a},P_{\leq a'}\big\}$
            for $\{a,a'\}$ in $\DDD(P)$.
    \end{itemize}
    \label{LemDPF}
\end{lemma}

\begin{proof}
Assertion (i) is equivalent to saying that,
in building up a forest with duplications, once
a duplication pair $\{a,a'\}$ is created from
duplicating a hanger $a$ in a poset
$P$, then neither $a$ nor $a'$ will ever be a hanger at some later
stage of the construction.  To see this, note that
any element $p$ in the nonempty poset $P_{<a}$ which is
covered by $a$ will also be covered by $a'$ after the duplication.
Thus in the new poset $P'$ after duplication, $p$ 
has a single-edge path to 
the element $a'$ of $P\setminus (P')_{\leq a}$ 
avoiding $a$, and similarly $p$ 
has a single-edge path to 
the element $a$ of $P\setminus (P')_{\leq a'}$ 
avoiding $a'$.  These single-edge
paths cannot be destroyed by any of the further
constructions, so neither $a$ nor $a'$ will ever be a hanger 
that is later duplicated.

We prove assertions (ii) and (iii) by induction
on the cardinality of $P$, that is, on the number
of operations used in constructing $P$.
It suffices to show that they remain true 
when performing any of the three construction operations.
This is trivial for the disjoint union construction,
and straightforward for the hanging construction.

For the duplication of a hanger operation, we argue more carefully.
Assume that $P'$ is obtained from the forest with duplications 
$P$ by duplicating the hanger $a$, to form a new pair
$\{a,a'\}$ with $\DDD(P') = \DDD(P) \sqcup \{\{a,a'\}\}$. 
We will make use of the order-preserving surjection
$\pi: P' \twoheadrightarrow P$ that collapses both $a$ and $a'$
to $a$.

For assertion (ii), note that $\pi$ sends any connected order ideal $J'$ 
in $\Jconn(P')$ to a connected order ideal $J:=\pi(J')$ 
in $\Jconn(P)$.   By induction, one knows that
$J$ is either of the form $J=P_{\leq p}$, or of the form
$P_{\leq b} \cup P_{\leq b'}$ where $\{b,b'\}$ lies in $\DDD(P)$.
It is now straightforward to check that
\begin{enumerate}
\item[$\bullet$]
if $J=P_{\leq p}$ for some $p \neq a$, then $J'=(P')_{\leq p}$,
\item[$\bullet$]
if $J=P_{\leq a}$, then $J'$ is either $(P')_{\leq a}$ or $(P')_{\leq a'}$
or $(P')_{\leq a} \cup (P')_{\leq a'}$, and
\item[$\bullet$]
if $J=P_{\leq b} \cup P_{\leq b'}$ where $\{b,b'\}$ lies in $\DDD(P)$,
then $J'=(P')_{\leq b} \cup (P')_{\leq b'}$.
\end{enumerate}
Thus $\Jconn(P')$ is exactly as described.

For assertion (iii), first note that $\{(P')_{\leq a}, (P')_{\leq a'}\}$ is a pair
of connected order ideals intersecting nontrivially, and hence lies in $\Pi(P)$.
Now assume $J'$ is in $\Jconn(P')$, but $J' \neq (P')_{\leq a}, (P')_{\leq a'}$.
We have seen above that $J'=\pi^{-1}(J)$ for some $J$ in $\Jconn(P)$.
If $J$ contains $a$, then $J'$ contains both $(P')_{\leq a}, (P')_{\leq a'}$,
and hence has trivial intersection with either of them. If $J$ does not contain $a$,
then since $a$ is a hanger in $P$, connectivity of $J$ forces it to lie entirely
in $P_{<a}$ or $P\setminus P_{\leq a}$, and will still have trivial intersection
with either of $P_{\leq a}, P_{\leq a'}$.  This analysis shows that the pairs
$\{J_1',J_2'\}$ in $\Pi(P')$ other than $\{(P')_{\leq a}, (P')_{\leq b'}\}$
are of the form $\{\pi^{-1}(J_1),\pi^{-1}(J_2)\}$ for some pair
$\{J_1,J_2\}$ in $\Pi(P)$.  By induction,
$\{J_1,J_2\}=\{P_{\leq b},P_{\leq b'}\}$ 
for some $\{b,b'\}$ in $\DDD(P)$, and then one can check that
$\{J'_1,J'_2\}=\{(P')_{\leq b},(P')_{\leq b'}\}$.
\end{proof}

The next result is the crux of Theorem~\ref{thm:duplicated-forest-extensions},
and will follow easily via inclusion-exclusion from 
Lemma~\ref{lemma:duplicated-forest-structure}.

\begin{theorem}
\label{thm:essence-of-factorization}
For a forest with duplications $P$ on $n$ elements,
one has
$$
\sum_{f \in \AAA^\weak(P)} t^{\nu(f)} \xx^f =
\frac{\prod_{\{J_1,J_2\} \in \Pi(P)} \left( 1-t^2 \xx^{J_1} \xx^{J_2} \right)}
     {\prod_{J \in \Jconn(P)} \left( 1-t \xx^J \right)}.
$$
Setting $t=1$ and $x_i=q$ for all $i$, this gives
\begin{equation}
\label{q-forest-with-duplication-gf}
\sum_{f \in \AAA^\weak(P)} q^{|f|} =
\frac{\prod_{\{J_1,J_2\} \in \Pi(P)}\left( 1- q^{|J_1|+|J_2|}\right)}
     {\prod_{J \in \Jconn(P)} \left( 1-q^{|J|}\right) }.
\end{equation}
\end{theorem}
\begin{proof}
Given a forest with duplications $P$,
we wish to evaluate
$
\sum_{f \in \AAA^\weak(P)} t^{\nu(f)} \xx^f,
$
where the sum runs over all weak $P$-partitions $f$.
By Proposition~\ref{P-partition-expressions-prop}, 
this is the same as the sum
$
\sum_{ \{J_i\}} \prod_{i} t\xx^{J_i}
$
over all multisubsets $\{J_i\}$ of $\Jconn(P)$
for which the $\{J_i\}$ pairwise intersect trivially.
By Lemma~\ref{lemma:duplicated-forest-structure}
this is equivalent to saying that the multiset $\{J_i\}$
contains no pair $\{P_{\leq a}, P_{\leq a'}\}$ with $\{a,a'\}$ in $\DDD(P)$.
Using inclusion-exclusion, this sum then equals
$$
\sum_{ \EEE \subseteq \DDD(P)  }  (-1)^{|\EEE|} 
    \sum_{ \{J_i\} } \prod_{i} t\xx^{J_i} 
$$
where the inside summation is over all
multisubsets $\{J_i\}$ of $\Jconn(P)$ that contain
at least the pair $\{ P_{\leq a},P_{\leq a'} \}$ for every
$\{a,a'\}$ in $\EEE$.  Finally, this can be rewritten
$$
\sum_{ \EEE \subseteq \DDD(P)  }  (-1)^{|\EEE|} 
     \frac{\prod_{\{a,a'\} \in \EEE} t\xx^{P_{\leq a}} \cdot t\xx^{P_{\leq a'}}}
       {\prod_{J \in \Jconn(P)} (1-t\xx^{J})} \\
=\frac{\prod_{\{a,a'\} \in \DDD(P)}  
          (1-t^2 \xx^{P_{\leq a}} \xx^{P_{\leq a'}}) }
       {\prod_{J \in \Jconn(P)} (1-t\xx^{J})}. 
$$
\end{proof}

\begin{proof}[Proof of 
  Theorem~\ref{thm:duplicated-forest-extensions}]
  Recall that for naturally labelled posets, weak $P$-partitions
  coincide with $P$-partitions.
    Then \eqref{duplicated-forest-q-formula} follows
    from \eqref{generating-function-in-q} 
    and \eqref{q-forest-with-duplication-gf}.
\end{proof}


\section{The rings and their Hilbert series}
\label{associated-graded-section}

We now change focus in the next few sections to 
discuss the weak $P$-partition ring $R_P$,
an example of a {\it normal affine semigroup ring}.  Good discussions of
general theory on affine semigroup rings may be 
found in Bruns and Herzog \cite[Chapter 6]{BrunsHerzog},
Miller and Sturmfels \cite[Chapter 7]{MillerSturmfels}, 
Stanley \cite[Chapter 1]{Stanley-CCA},
and Sturmfels \cite{Sturmfels}.

\begin{definition}
For $P$ a poset on $\{1,2,\ldots,n\}$,
let $R_P$ be the subalgebra of the polynomial ring $k[x_1,\ldots,x_n]$
which is spanned $k$-linearly by the monomials
$$
\xx^f:=x_1^{f(1)} \cdots x_n^{f(n)}
$$ 
as $f$ runs through all weak $P$-partitions.  
In \cite{BFLR} this was the ring denoted $R^{\mathrm{wt}}_P$.

 Let $\mm$ denote the 
maximal ideal of $R_P$ spanned $k$-linearly by all monomials $\xx^f$ with 
$f\neq 0$, so that $R_P/\mm \cong k$.  As usual, one has
the {\it $\mm$-adic} filtration 
\begin{equation}
\label{maximal-ideal-filtration}
R_P \supset \mm \supset \mm^2 \supset \mm^3 \supset \cdots
\end{equation}
and the {\it associated graded ring}
$$
\gr(R_P):= R_P /\mm \oplus \mm/\mm^2 \oplus  \mm^2/\mm^3 \oplus \ldots.
$$
In this ring $\gr(R_P)$, multiplication is defined $k$-linearly
by saying that the product of two elements
$\bar{f}$ in $\mm^i/\mm^{i+1}$ and $\bar{g}$ 
$\mm^j/\mm^{j+1}$ is $\overline{fg}$ in $\mm^{i+j}/\mm^{i+j+1}$.
\end{definition}

Note that $R_P$ has a natural $\NN^n$-multigrading, in which the
degree of $\xx^f$ is $(f(1),\ldots,f(n)) \in \NN^n$.  Then its
$\NN^n$-graded {\it Hilbert series} will be
$$
\Hilb(R_P, \xx) = \sum_{f \in \AAA^\weak(P)} \xx^f,
$$
that is
the same generating function\footnote{Again assuming that
$P$ has been {\it naturally labelled}; 
see Remark~\ref{natural-labelling-remark}.} 
that appears in \eqref{generating-function-in-x}.

Note also that $\gr(R_P)$ enjoys this same $\NN^n$-multigrading,
and even the same $\NN^n$-graded {\it Hilbert series} as $R_P$, since the
the $\mm$-adic filtration \eqref{maximal-ideal-filtration} is a filtration by
$\NN^n$-homogeneous ideals. 

We will always use the $\xx$-variable set for the power series that are
Hilbert series with respect to this $\NN^n$-multigrading.
In addition, one can collapse the $\NN^n$-multigrading to
an $\NN$-grading by letting $x_i=q$ for all $i$.
We will use the variable $q$ for power series
which are Hilbert series for this grading.

Furthermore, $\gr(R_P)$ has its standard $\NN$-grading
in which its homogeneous component of degree $i$ is $\mm^i/\mm^{i+1}$.
We call the {\it $t$-grading} and use the variable $t$ 
in the corresponding Hilbert series.

In fact, one can form an even finer Hilbert series
$
\Hilb(\gr(R_P), t, \xx)
$
that keeps track of both the $t$-grading and the $\NN^n$-multigrading.
We will see shortly that this series
is exactly the right side of \eqref{generating-function-in-t,x}.

Proposition \ref{P-partition-expressions-prop} (iii) has the
following consequence.  Fixing a field $k$, introduce a 
polynomial algebra $S=k[U_J]$ having generators
$U_J$ indexed by connected order ideals $J$ of $P$.  
For the sake of considering multigraded maps,
consider $S$ as $\NN^n$-multigraded, with the variable $U_J$ having
the same degree as the monomial $\xx^J$, namely the characteristic
vector $\chi_J$ in $\NN^n$.
In particular, when we collapse the grading into an $\NN$-grading,
the variable $U_J$ has degree $|J|$.
In addition, $S$ admits another interesting $\NN$-grading, where all $U_J$
have degree $1$, corresponding to the $t$-grading discussed earlier.

\begin{corollary}(cf. \cite[Proposition 7.1]{BFLR})
\label{cor:minimal-ring-generators}
The ring $R_P$ is minimally generated as a $k$-algebra
by the monomials $\xx^J$ as $J$ runs through $\Jconn(P)$.
In particular, these maps
$$
\begin{array}{rclcrcl}
S &\overset{\varphi}{\longrightarrow} & R_P &\text{ and }&
S &\overset{\gr(\varphi)}{\longrightarrow} &\gr(R_P)\\
U_J &\longmapsto & \xx^J &  &U_J &\longmapsto &\bar{\xx}^J.
\end{array}
$$
are multigraded $k$-algebra surjections with respect to the $\NN^n$-gradings.  
Moreover, the second map is also $\NN$-graded with respect to the $t$-gradings.
\end{corollary}
\begin{proof}
The fact that $\{\xx^J\}_{J \in \Jconn(P)}$ minimally generate $R_P$ 
was proven in \cite[Proposition 7.1]{BFLR}, but
we repeat the proof here for completeness.

The fact that they generate $R_P$ follows
from Proposition \ref{P-partition-expressions-prop} (iii).
Their {\it minimality} follows from the claim that
the characteristic vectors $\chi_J$ for $J$ in $\Jconn(P)$
are exactly the set of primitive vectors spanning the 
extreme rays of the real cone nonnegatively spanned 
by the $P$-partitions\footnote{In \cite[Proposition 4.6.10]{Stanley-EC} such vectors are called the {\it completely fundamental} elements of the semigroup.}.

To see this claim, given $J$ in $\Jconn(P)$, consider 
the Hasse diagram for $J$ as a connected graph, and pick a
spanning tree $T$ among its edges. Then the line 
$\RR \chi_J$ is exactly the intersection of the hyperplanes
$x_i = 0$ for $i \not\in J$, and
$x_i=x_j$  for $\{i,j\}$ an edge of $T$.
All of these hyperplanes arise as cases of equality in various
half-space inequalities that define the weak $P$-partition
cone.  Hence each such $\chi_J$ spans an extreme ray of the cone.

Since $\{ \xx^J\}_{J \in \Jconn(P)}$
is a minimal generating set for $R_p$ as an algebra,
their images $\{ \overline{\xx}^J\}_{J \in \Jconn(P)}$ by $\gr(\varphi)$
give a $k$-basis for $\mm/\mm^2$ .
Hence each such element has $t$-degree $1$ and so
the map $\gr(\varphi)$ respects the $t$-grading.
\end{proof}

This result allows us to interpret combinatorially the power of $t$
in the power series
$
\Hilb(\gr(R_P), t, \xx)
$
and to obtain some information about its form.

\begin{corollary}
\label{associated-graded-Hilbert-series-corollary}
Let $P$ be any poset on $\{1,2,\ldots,n\}$.
\begin{enumerate}
\item[(i)] 
The $\NN \times \NN^n$-graded Hilbert series
for $\gr(R_P)$ is given by
$$
\Hilb(\gr(R_P),t,\xx) = \sum_{f \in \AAA^\weak(P)} t^{\nu(f)} \xx^f.
$$
\item[(ii)]
The power series in (i) can always be expressed as
$$
\frac{g(t,\xx)}{\prod_{J \in \Jconn(P)} \left( 1-t \xx^J \right)}
$$
for some polynomial $g(t,\xx)$ in $\ZZ[t,\xx]$.
\item[(iii)]
Furthermore, the
generating functions appearing in Corollary~\ref{generating-functions-corollary}
are the Hilbert series for $R_P$ or $\gr(R_P)$ with respect to
their $\NN \times \NN^n$-grading or $\NN \times \NN$-grading or
$\NN^n$ or $\NN$-grading, where appropriate.
\end{enumerate}
\end{corollary}
\begin{proof}
For assertion (i), note that $\overline{\xx}^J$ has $t$-degree $1$ 
and $\NN^n$-multidegree $\chi_J$ in $\gr(R_P)$.  This means
that if $f=\sum_{i=1}^{\nu(f)} \chi_{J_i}$ for connected order ideals
$J_i$, then $\overline{\xx}^f = \prod_{i=1}^{\nu(f)} \overline{\xx}^{J_i}$ 
will have $t$-degree $\nu(f)$ and $\NN^n$-multidegree $f$,
as desired.

For assertion\footnote{An alternate argument for assertion (ii) is
to apply \cite[Prop. 4.6.11]{Stanley-EC}.} (ii), note that $\gr(R_P)$ becomes
a finitely-generated $\NN \times \NN^n$-graded
$S$-module where $S=k[U_J]_{J \in \Jconn(P)}$.  It therefore
has an $\NN \times \NN^n$-graded free $S$-resolution,
$$
0 \rightarrow F_\ell \rightarrow \cdots \rightarrow 
F_1 \rightarrow F_0 \rightarrow \gr(R_P) \rightarrow 0,
$$
with $F_0=S$, and whose length $\ell$ is guaranteed to be at most
$|\Jconn(P)|$ by Hilbert's Syzygy Theorem.  Letting
$\beta_{i,(j,\alpha)}$ denote the number of $S$-basis elements
of the free $S$-module $F_i$ having
$\NN \times \NN^N$-multidegree $(j,\alpha)$,
then
$$
\begin{aligned}
\Hilb(R_P,t,\xx) 
&= \Hilb(S,t,\xx) \cdot
\sum_{i=0}^\ell (-1)^i \sum_{(j,\alpha) \in \NN \times \NN^n} 
                         \beta_{i,(j,\alpha)} t^j \xx^\alpha \\
&={\sum_{\substack{i=0,1,\ldots,\ell\\ (j,\alpha)  \in \NN \times \NN^n} } 
    \beta_{i,(j,\alpha)} (-1)^i t^j \xx^\alpha}  \Bigg/
   {\prod_{J \in \Jconn(P)} \left( 1-t \xx^J \right)}.
\end{aligned}
$$
Thus the numerator here is the polynomial $g(t,\xx)$.
\end{proof}

\section{Presentations and 
proofs of Theorems~\ref{thm:minimal-presentation}
and \ref{thm:two-other-ideal-generators}}
\label{presentations-section}

Here we analyze further the structure of the rings $R_P$ and $\gr(R_P)$, by means of 
the surjections $\varphi$ and $\gr(\varphi)$ from Corollary~\ref{cor:minimal-ring-generators}.

\begin{definition}
\label{ideal-definitions}
Define three ideals within the polynomial ring $S=k[U_J]_{J \in \Jconn(P)}$
each with generating sets indexed by the set $\Pi(P)$ that consists of
all pairs $\{J_1, J_2\}$ of connected order ideals in $P$ which intersect nontrivially:
$$
\begin{aligned}
I_P&:=( \syz_{J_1,J_2} )_{\{J_1, J_2\} \in \Pi(P)} \\
I^\gr_P&:= ( \syz^\gr_{J_1,J_2} )_{\{J_1, J_2\} \in \Pi(P)} \\
I^\init_P&:= ( \syz^\init_{J_1,J_2} )_{\{J_1, J_2\} \in \Pi(P)}
\end{aligned}
$$
where $\syz_{J_1,J_2}, \syz^\gr_{J_1,J_2}, \syz^\init_{J_1,J_2}$ 
were defined in \eqref{P-partition-syzygy}, \eqref{gr-syzygy}, 
and \eqref{monomial-syzygy} in the Introduction.
\end{definition}

We will see further (Proposition \ref{presentations-prop})
that $I_P$ and $I^\gr_P$ are the kernels of
the morphisms $\varphi$ and $\gr(\varphi)$, so that
$R_P \simeq S / I_P$ and $\gr(R_P) \simeq S / I^\gr_P$.
We first establish a link between these rings and $S/I^\init_P$.

\begin{proposition}
\label{shared-Hilbert-series-prop}
For any $P$ on $\{1,2,\ldots,n\}$,
the three rings 
$$
\begin{aligned}
& R_P \\
& \gr(R_P) \\ 
& S/I^\init_P
\end{aligned}
$$ 
share the same $\NN^n$-graded Hilbert series, namely $\sum_{f \in \AAA^\weak(P)} \xx^f$.
\end{proposition}
\begin{proof}
By definition, $R_P$ has this sum $\sum_f \xx^f$ as
its $\NN^n$-graded Hilbert series.  Setting $t=1$ in
Corollary~\ref{associated-graded-Hilbert-series-corollary}(i)
show that the same for $\gr(R_P)$.
Finally, Proposition \ref{P-partition-expressions-prop} part (ii)
implies that $S/I^\init_P$ 
also has this same generating function 
as its $\NN^n$-graded Hilbert series, since the monomials 
surviving in the quotient $S/I^\init_P$ correspond to multisets of
nonempty connected order ideals that pairwise intersect trivially.
\end{proof}

The relation between the monomial quotient $S/I^\init_P$ and the rings $R_P$ and $\gr(R_P)$
is in fact deeper than an equality of Hilbert series.
Indeed, it fits into the theory of
Gr\"obner bases (see, e.g., Sturmfels \cite[Chapter 1]{Sturmfels}).
Recall that a {\it monomial ordering} on $S$ is a total ordering $\preceq$ on the set of 
all monomials $U^A$ in $S$ with these properties:
\begin{enumerate}
\item[(a)] $\preceq$ has no infinite descending
chains, 
\item[(b)] the monomial $1=U^0$ is the smallest element for $\preceq$, and
\item[(c)] for any monomials $U^A, U^B, U^C$, 
$$
U^A \preceq U^B \text{ implies }
U^A U^C \preceq U^B U^C.
$$
\end{enumerate}
Having fixed a  monomial ordering $\preceq$,
given a polynomial $f$ in $S$, its {\it initial term} $\init_\preceq(f)$
is its monomial with nonzero coefficient
which is highest in the $\preceq$ order.
Given an ideal $I \subset S$, its {\it initial ideal} is the monomial
ideal $\init_\preceq(I):=(\init_\preceq(f) )_{f \in I}$.

Given a poset $P$, we define a total 
ordering $\preceq$ on the monomials in $S$ as follows.
First choose a total order $\preceq$ on order ideals of $P$ 
such that $|J| < |K|$ implies $J \prec K$.
Then when comparing two distinct monomials 
$$
\begin{array}{rcccl}
\UU_J&=&U_{J_1} U_{J_2} \cdots U_{J_r} 
&\text{ with }&J_1 \preceq J_2 \preceq \dots \preceq J_r, \\
\UU_K&=& U_{K_1} U_{K_2} \cdots U_{K_s}
&\text{ with }&K_1 \preceq K_2 \preceq \dots \preceq K_s,
\end{array}
$$
assume without loss of generality that $r \leq s$.
Find the smallest $i$ in $\{1,2,\ldots, r\}$
for which $J_i \neq K_i$; if no such $i$ exists, so 
$\UU_J$ strictly divides $\UU_K$,
say that $\UU_J \prec \UU_K$.
Otherwise, if $J_i \prec K_i$
say that $\UU_J \prec \UU_K$,
and if $K_i \prec J_i$ say that $\UU_K \prec \UU_J$.
It is not hard to see that such a linear order
$\preceq$ will satisfy the above properties (a),(b),(c) that
define a monomial ordering.

Theorems~\ref{thm:minimal-presentation} and 
\ref{thm:two-other-ideal-generators} amount to the following result.

\begin{theorem}
\label{presentations-prop}
For a poset $P$ on $\{1,2,\ldots,n\}$, one has these ideal equalities:
$$
\begin{aligned}
I_P&=\ker\left(\varphi:S \longrightarrow R_P\right) \\
I^\gr_P&= \ker\left(\gr(\varphi):S \longrightarrow \gr(R_P)\right)\\
I^\init_P& = \init_\preceq(I_P) =\init_\preceq(I^\gr_P)
\end{aligned}
$$
where $\preceq$ is a monomial order on $S$ defined
as above.
\end{theorem}
The first equality asserts that
$I_P$ is the toric ideal for the ring $R_P$
with respect to its minimal generating set, in the terminology
of Sturmfels \cite{Sturmfels}. 
\begin{proof}
Temporarily denote by $K, K^\gr$ the 
kernels appearing on the right sides in the theorem:
$$
\begin{aligned}
K &:=\ker( \varphi: S \longrightarrow R_P ); \\
K^\gr &:=\ker( \gr(\varphi): S \longrightarrow \gr(R_P) ).
\end{aligned}
$$
One can check from the generators of $I_P$ and $I^\gr_P$ 
given in Definitions~\ref{ideal-definitions}
that $I_P \subseteq K$ and $I^\gr_P \subseteq K^\gr$.
Hence one has inclusions
$$
\begin{aligned}
\init_\preceq(I_P) &\subseteq \init_\preceq(K) \\
\init_\preceq(I^\gr_P) &\subseteq \init_\preceq(K^\gr).
\end{aligned}
$$
On the other hand, since 
$$
\syz^\init_{J_1,J_2}
=U_{J_1} U_{J_2} 
=\init_\preceq(\syz_{J_1,J_2}) 
=\init_\preceq(\syz^\gr_{J_1,J_2})
$$
one concludes that
$$
I^\init_P \subseteq \init_\preceq(I_P), 
\init_\preceq(I^\gr_P).
$$
These various ideal inclusions lead to towers of surjections
\begin{equation}
\label{towers-of-surjections}
\begin{array}{rcl}
S/I^\init_P 
& \twoheadrightarrow 
S/\init_\preceq(I_P) 
& \twoheadrightarrow 
S/\init_\preceq(K) 
\\
S/I^\init_P 
& \twoheadrightarrow 
S/\init_\preceq(I^\gr_P) 
& \twoheadrightarrow 
S/\init_\preceq(K^\gr) 
\end{array}
\end{equation}
\noindent
Recall that for any homogeneous ideal $I$ of $S$ and any
monomial ordering $\preceq$, the initial ideal $\init_\preceq(I)$
has the property that $S/I$ and $S/\init_\preceq(I)$ share the
same Hilbert series.  Together with 
Proposition~\ref{shared-Hilbert-series-prop} this shows 
all these quotient rings
$$
\begin{array}{rll}
&S/K & \left( \cong R_P \right) \\
&S/K^\gr &\left( \cong \gr(R_P) \right) \\
&S/I^\init_P\\
&S/\init_\preceq(K) \\
&S/\init_\preceq(K^\gr) \\
\end{array}
$$
share the same $\NN^n$-multigraded Hilbert series.
One concludes that all of the surjections in the towers \eqref{towers-of-surjections} 
are isomorphisms.  Thus
$$
\begin{array}{rll}
I^\init_P &= \init_\preceq(I_P) &= \init_\preceq(K)\\
I^\init_P &= \init_\preceq(I^\gr_P) &= \init_\preceq(K^\gr).\\
\end{array}
$$
and the generators for $I_P, I^\gr_P$ given in their definitions form
Gr\"obner bases with respect to $\preceq$ for the ideals $K, K^\gr$.
This implies $I_P=K$ and $I^\gr_P=K^\gr$.
\end{proof}

\begin{proposition}
\label{toric-minimal-generators-prop}
Each of the three ideals $I_P, I^\gr_P$ and $I^\init_P$
is generated minimally by the generating sets appearing
in Definition~\ref{ideal-definitions} indexed by $\Pi(P)$.
\end{proposition}
\begin{proof}
We give the argument by contradiction for why the
generator 
$$
\syz_{J_1,J_2}=U_{J_1} U_{J_2} - U_{J_1 \cup J_2} \prod_{i=1}^t U_{J^{(i)}}
$$
of $I_P$ cannot be redundant;  the arguments for
$I^\gr_P$ and $I^\init_P$ are similar and even easier.
  If $\syz_{J_1,J_2}$ were redundant, then it could be expressed as a sum 
$$
\syz_{J_1,J_2} =
\sum_{\substack{ \{K_{1},K_{2}\}\  \in \  \Pi_{P} \\
           \{K_{1},K_{2}\} \neq \{J_{1},J_{2}\} } } 
      f_{K_{1},K_{2}} \cdot \syz_{K_{1},K_{2}}
$$
where the $f_{K_{1},K_{2}}$ are some polynomials in the
variables $U_J$ of $S$.  Since the monomial $U_{J_1} U_{J_2}$ appears
on the left, it must appear in the right, say in the term 
$f_{K_{1},K_{2}} \cdot \syz_{K_{1},K_{2}}$, forcing
one of the two monomials $U_{K_1} U_{K_2}$ or 
$U_{K_1 \cup K_2} \prod_{i=1}^m U_{K^{(i)}}$ 
in $\syz_{K_{1},K_{2}}$ to divide $U_{J_1} U_{J_2}$.
Since $U_{J_1} U_{J_2}$ is quadratic, this forces
either the equality of sets 
\begin{enumerate}
\item[$\bullet$]
$\{J_1,J_2\}=\{ K_1, K_2\}$, a contradiction, or 
\item[$\bullet$]
$m=1$ (that is, $K_{1} \cap K_{2}=K^{(1)}$ is connected) 
and $\{K_{1} \cup K_{2}, K_{1} \cap K_{2} \} = \{J_{1},J_{2}\}$.
This is again a contradiction because $J_{1}$ and $J_{2}$ 
have non-trivial intersection, that is, neither one is included
in the other.\qedhere
\end{enumerate}
\end{proof}

We close this section by discussing the
situation when $\gr(R_P) \cong R_P$.

\begin{corollary}
\label{gradedness-characterization}
The following are equivalent for a poset 
$P$ on $\{1,2,\ldots,n\}$:
\begin{enumerate}
\item[(i)]
One has $I^\gr_P=I_P$ and $\gr(R_P) \cong R_P$.
\item[(ii)]
The toric ideal $I_P=\ker(S \overset{\varphi}{\rightarrow} R_P)$ 
is homogeneous for the standard $\NN$-grading on $S$ in which
each $U_J$ has degree one.
\item[(iii)]
Every pair $\{J_1,J_2\}$ of connected order ideals that
intersects nontrivially has $J_1 \cap J_2$ connected.
\end{enumerate}
\end{corollary}
\begin{proof}
The equivalence of (i) and (ii) is easy and well-known.
For the equivalence of (ii) and (iii), note
that the minimal generator $\syz_{J_1,J_2}$
is homogeneous if and only if $t=1$, that is, if
and only if $J_1 \cap J_2$ is connected.  Now apply 
Proposition~\ref{toric-minimal-generators-prop}.
\end{proof}

An important special case of this situation where $\gr(R_P) \cong R_P$
was studied by Hibi \cite{Hibi}, namely when {\it every} nonempty order ideal
is connected.  We leave the straightforward proof of the following proposition
to the reader.

\begin{proposition}
A finite poset $P$ has every nonempty order ideal connected 
if and only if $P$ has a minimum element $\hat{0}$.  Furthermore,
in this case,
\begin{enumerate}
\item[$\bullet$]
the two decompositions in
Proposition~\ref{P-partition-expressions-prop} (ii) and (iii) coincide,
\item[$\bullet$]
the statistic $\nu(f)$ on $P$-partitions $f$ equals the maximum value 
$\max(f)$, 
\item[$\bullet$]
the statistic $\des_P(w)$ on linear extensions $w$ in $\LLL(P)$
is independent of the poset structure $P$,
and equals the {\it descent number} $\des(w):=|\Des(w)|$,
\item[$\bullet$]
the ring $R_P \cong \gr(R_P)$ is the same
as the Hibi ring introduced in \cite{Hibi}, 
but associated with the poset $P \setminus \hat{0}$.
In other words, 
$$
\begin{aligned}
R_P & \cong \gr(R_P) \\
    & \cong k[y_J]_{J \in \JJJ(P \setminus \hat{0})} \,\,\  / \,\,\,  
\left( \,\,\, y_{J_1}\cdot y_{J_2} - y_{J_1 \cup J_2} \cdot y_{J_1 \cap J_2} \,\,\, \right)_{J_1, J_2 \in \JJJ(P)}.
\end{aligned}
$$
\end{enumerate}
\end{proposition}

\section{Second proof of Theorem~\ref{thm:duplicated-forest-extensions}: 
complete intersections}
\label{second-forest-proof-section}

We give here a second proof, via our ring presentations,
of the precursor Theorem~\ref{thm:essence-of-factorization},
rather than Theorem~\ref{thm:duplicated-forest-extensions} 
itself.\medskip

This proof uses some basic notions of commutative algebra that we shall 
recall here: we refer to Stanley \cite[\S I.5]{Stanley-CCA} for
more details on this subject.

The Krull dimension $\Kdim(A)$ of a finitely generated commutative
$k$-algebra $A$ is the maximum cardinality $d$ of a 
subset $\{ \theta_1,\ldots,\theta_d \}$
in $A$ which are algebraically independent over $k$.
If $A$ is $\NN$-graded, then the Krull dimension 
coincides with the multiplicity
of the pole $z=1$ in the Hilbert series $\Hilb(A,z)$.
In particular, when several algebras share the same Hilbert series,
they also share the same Krull dimension.

Let $\theta_1,\dots,\theta_\ell$ be homogeneous elements in a graded 
$k$-algebra $A$.
Then one has the inequality
\begin{equation}
\label{Hauptidealsatz}
\Kdim \big( A / (\theta_1 A + \dots + \theta_\ell A) \big) 
\geq \Kdim A - \ell.
\end{equation}
If $A$ is \emph{Cohen-Macaulay} (which is for example the case of
a polynomial algebra over a field), then equality in
\eqref{Hauptidealsatz}
occurs if and
only if for each $i=1,2,\ldots,\ell$ one has
that $\theta_i$ is a non-zero-divisor in the quotient
$A / (\theta_1 A + \dots + \theta_{i-1} A).$
Such a sequence $(\theta_1,\ldots,\theta_\ell)$
is called an {\it $A$-regular sequence}.

We now have all the necessary tools to give our second proof
of Theorem~\ref{thm:essence-of-factorization}.

\begin{proof}[Proof of 
Theorem~\ref{thm:essence-of-factorization}]
For any poset $P$ on $\{1,2,\ldots,n\}$,
the affine semigroup ring $R_P$ of $P$-partitions
has Krull dimension $n$, since the cone of $P$-partitions
is $n$-dimensional.  But then $\gr(R_P)$ and $S/I^\init_P$
also have Krull dimension $n$, since  
Proposition~\ref{shared-Hilbert-series-prop}
asserts that they have the same $\NN^n$-graded Hilbert series.

Now the presentation for any of the three rings
$R_P, \gr(R_P), S/I^\init_P$
in Theorems~\ref{thm:minimal-presentation} and
\ref{thm:two-other-ideal-generators} exhibits them
as quotients of $S=k[U_J]_{J \in \Jconn(P)}$, which
has Krull dimension $|\Jconn(P)|$,
by an ideal ($I_P, I^\gr_P$ or $I^\init_P$)
having $|\Pi(P)|$ minimal generators.
Hence, one always has the inequality 
\begin{equation}
\label{codimension-inequality}
|\Jconn(P)| - |\Pi(P)| \geq n
\end{equation}
and equality occurs if and only if this is a {\it complete intersection
presentation}, meaning that the ideal generators
in each case form an $S$-regular sequence.

When these are complete intersection presentations,
one obtains the following Hilbert series calculation 
for $\gr(R_P)$
$$
\left( \sum_{f \in \AAA^\weak(P)} t^{\nu(f)} \xx^f = \right)
\Hilb(\gr(R_P),t,\xx) =
\frac{\prod_{\{J_1,J_2\} \in \Pi(P)} \left( 1-t^2 \xx^{J_1} \xx^{J_2} \right)}
     {\prod_{J \in \Jconn(P)} \left( 1-t \xx^J \right)}.
$$
by iterating the relation 
$$
\Hilb(R/(\theta),t) = (1-t^{\deg(\theta)}) \cdot \Hilb(R,t)
$$
which holds for a nonzero divisor $\theta$ in a (multi-)graded ring $R$;
see \cite[\S I.5, page 25]{Stanley-CCA}.
It only remains to note that when $P$ is a forest with duplications,
Proposition~\ref{forest-with-duplication-ideal-description} shows
$|\Jconn(P)| = n + |\DDD(P)|$ and $|\Pi(P)| = |\DDD(P)|$.
Equality in \eqref{codimension-inequality} follows.
\end{proof}

\section{Koszulity}
\label{koszulity-section}

We discuss here an immediate consequence of $I^\gr_P$ having
a quadratic initial ideal $I^\init_P$, coming from the theory
of {\it Koszul algebras}. The reader is referred to Fr\"oberg \cite{Froeberg-survey}
and the book by Polishchuk and Positselski \cite{PolishchukPositselski}
for background on Koszul algebras.

\begin{corollary}
For any poset $P$ on $\{1,2,\ldots,n\}$,
the graded ring $A=\gr(R_P)$ is a Koszul algebra.
In other words, $(R_P, \mm)$ is nongraded Koszul in the sense considered
by Fr\"oberg \cite{Froeberg}.

In particular, the $\NN \times \NN^n$-multigraded
Hilbert series $\Hilb(A, t, \xx)$
described in Corollary~\ref{associated-graded-Hilbert-series-corollary}
has the property that $\Hilb(A, -t, \xx)^{-1}$ 
lies in $\NN[t,\xx]$, as it is the Hilbert series
for the Koszul dual algebra $A^!$.
\end{corollary}
\begin{proof}
It is well-known (see e.g., \cite[Prop. 3]{EisenbudReevesTotaro})
that having an initial ideal generated by quadratic monomials,
as is the case with $I^\init_P=\init_{\preceq}(I^\gr_P)$,
suffices to imply Koszulity.  The relation between
the Hilbert series of a Koszul ring $A$ and its Koszul dual $A^!$ is
also standard.
\end{proof}

\begin{example}
Since Theorem~\ref{thm:essence-of-factorization} 
implies that a forest with duplications $P$ has
$$
\Hilb(R_P,t,\xx) =
\frac{\prod_{\{J_1,J_2\} \in \Pi(P)} \left( 1-t^2 \xx^{J_1} \xx^{J_2} \right)}
     {\prod_{J \in \Jconn(P)}\left( 1-t \xx^J \right)}.
$$
one sees that
$$
\Hilb(R_P,-t,\xx)^{-1} =
\frac{\prod_{J \in \Jconn(P)} \left( 1+t \xx^J \right)}
{\prod_{\{J_1,J_2\} \in \Pi(P)} \left( 1-t^2 \xx^{J_1} \xx^{J_2} \right)}
$$
which manifestly lies in $\NN[t,\xx]$.
\end{example}

\begin{example}
The naturally labelled poset $P$ from Example~\ref{ex:five-element-example} 
had $\Hilb(R_P,t,\xx)$ equal to
$$
\frac{1-t^2 (\xx^{(1,2,1,1,0)} + \xx^{(1,2,1,1,1)} + \xx^{(2,2,2,1,1)})
+ t^3 (\xx^{(2,3,2,1,1)} + \xx^{(2,3,2,2,1)})}
{\prod_{J \in \Jconn(P)} \left( 1-t \xx^J \right)}
$$
and hence $\Hilb(R_P,-t,\xx)^{-1}$ equal to
$$
\frac
{\prod_{J \in \Jconn(P)} \left( 1+t \xx^J \right)}
{1-\left( t^2 (\xx^{(1,2,1,1,0)} + \xx^{(1,2,1,1,1)} + \xx^{(2,2,2,1,1)})
+ t^3 (\xx^{(2,3,2,1,1)} + \xx^{(2,3,2,2,1)}) \right)}
$$
which again manifestly lies in $\NN[t,\xx]$.
\end{example}

\section{The ideal of $P$-partitions, and the maj formula
for forests}
\label{non-natural-labelings-section}

When the poset $P$ is not naturally labelled, the
$P$-partitions $\AAA(P)$ form a proper subset
of the affine semigroup $\AAA^\weak(P)$ of weak
$P$-partitions.  In fact, this subset $\AAA(P)$
is a {\it semigroup ideal}, in the sense that
$$
\AAA^\weak(P)+ \AAA(P)=\AAA(P).
$$

\begin{definition}
For a poset $P$ on $\{1,2,\ldots,n\}$, let
$\III_P \subset R_P$ denote the ideal of the
affine semigroup ring $R_P$
spanned $k$-linearly by the monomials $\xx^f$
where $f$ runs through $\AAA(P)$.
\end{definition}

From the $R_P$-module filtration
$
\III_P \supset \mm \III_P \supset \mm^2 \supset \cdots
$
one can form the {\it associated $\gr(R_P)$-graded module}
$$
\gr(\III_P) = \III_P/\mm \III_P \oplus  \mm \III_P/\mm^2 \III_P 
                 \oplus  \mm^2 \III_P/\mm^3 \III_P \oplus \cdots .
$$

Recall that Corollary~\ref{cor:minimal-ring-generators} showed
that $R_P, \gr(R_P)$, respectively, were generated as $k$-algebras by
the collection of monomials $\{ \xx^J \}_{J \in \Jconn(P)}$ and their images within
$\mm/\mm^2$, respectively.  
Similarly, Proposition~\ref{fundamental-P-partition-proposition}
shows that the ideal $\III_P$ within $R_P$ is finitely generated,
by the monomials
\begin{equation}
\label{non-minimal-generators}
\left\{ \prod_{i \in \Des(w)} \xx^{w|_{[1,i]}} : w \in \LLL(P) \right\},
\end{equation}
and hence their images within $\III_P/\mm \III_P$ will
generate $\gr(\III_P)$ as a $\gr(R_P)$-module.  
As it is finitely generated, we can deduce the following result exactly as
in Corollary~\ref{associated-graded-Hilbert-series-corollary}.

\begin{corollary}
\label{cor:ideal-denominator-prediction}
The $\NN \times \NN^n$-graded Hilbert series for $\gr(\III_P)$ is
$$
\Hilb(\gr(\III_P),t,\xx)=
\sum_{f \in \AAA(P)} t^{\nu(f)} \xx^f
$$
and can always be expressed in the form
$$
\frac{h(t,\xx)}{\prod_{J \in \Jconn(P)} \left( 1-t \xx^J \right) }
$$
for some polynomial $h(t,\xx)$ in $\ZZ[t,\xx]$.
\end{corollary}

\begin{remark}
\label{rmk:non-minimal-generators-remark}
Note that
the monomials in \eqref{non-minimal-generators} will not
necessarily generate $\III_P$ {\it minimally} in general.  For example,
let $P=P_3$ be the poset with order relations $3<_P 1,2$ 
among those in Example~\ref{ex:three-little-posets}.  Then 
$$
\LLL(P)=\{\quad 3 \cdot 12, \qquad 3 \cdot 2 \cdot 1 \quad \}
$$
where here dots have been added indicating descents.
The generating set for $\III_P$ described in
\eqref{non-minimal-generators} is in this case $\{x_3, \, \quad x_3 \cdot x_2 x_3\}$.
However, it is easy to check (or see Proposition~\ref{prop:principal-ideal} below) 
that in this case $\III_P$ is the principal ideal within $R_P=k[x_3,x_1x_3,x_2x_3,x_1x_2x_3]$ generated
by the single monomial $\{x_3\}$.
\end{remark}

Although we do not know a {\it minimal} generating
set in general for the ideal $\III_P$ within $R_P$, it turns out to be 
easy to characterize when $\III_P$ is {\it principal},
that is, generated by a single element.  This is equivalent
to the existence of a minimum $P$-partition $f_{\min}$ in $\AAA(P)$
with the property that
$$
f_{\min} + \AAA^\weak(P) = \AAA(P)
$$
Such a characterization was provided by
Stanley (see \cite[Lemma 4.5.12]{Stanley-EC})
in the special case where $P$ is
{\it strictly labelled};  we explain here the obvious
modification of his characterization for the general case.  

\begin{definition}
We define a candidate for $f_{\min}$, the function
$
\delta: P \rightarrow \NN
$
whose value $\delta(i)$ is the maximum over all
saturated chains in $P_{\geq i}$ of the number
of {\it strict} covering relations in the chain,
that is, covering relations $i \lessdot_P j$
for which $i >_\NN j$.  It is easily checked both
that 
\begin{enumerate}
\item[(a)]
$\delta$ lies in $\AAA(P)$, and 
\item[(b)]
every $f$ in $\AAA(P)$ has $f(i) \geq \delta(i)$ for
all $i$.
\end{enumerate}

\noindent
Say that the poset $P$ on $\{1,2,\ldots,n\}$ satisfies
the {\it labelled-$\delta$-chain condition}\footnote{The reason
for this terminology is that, in the special case where $P$ is 
strictly labelled, it was called the {\it $\delta$-chain condition}
by Stanley in \cite[\S 4.5]{Stanley-EC}.} if for every $i$, 
all saturated chains in $P_{\geq i}$ have the same number
of strict covering relations, namely $\delta(i)$.
\end{definition}

\begin{proposition}
\label{prop:principal-ideal}
The $P$-partition ideal $\III_P$ within the
(weak) $P$-partition ring $R_P$ is a principal ideal 
if and only if $P$ satisfies the labelled-$\delta$-chain condition.
Furthermore, in this case $f_{\min} = \delta$.
\end{proposition}
\begin{proof}
The second assertion follows from properties
(a) and (b) above:  if $f_{\min}$ exists, then 
(a) implies $\delta \geq f_{\min}$, while
(b) implies $f_{\min} \geq \delta$.

For the first two assertions, note that the values of $\delta$ satisfy
\begin{equation}
\label{recursive-delta-definition}
\delta(i) 
\begin{cases}
=0 & \text{ if }i\text{ is maximal in }P\\
\geq \delta(j) & \text{ if }i \lessdot_P j
                 \text{ and }i >_\NN j\text{ for some }j \\
\geq \delta(j)+1 & \text{ if }i \lessdot_P j
                 \text{ and }i >_\NN j\text{ for some }j. \\
\end{cases}
\end{equation}
It is then easily seen that the labelled-$\delta$-chain
condition is equivalent to the assertion that
changing the inequalities in \eqref{recursive-delta-definition}
to equalities gives a well-defined recursive formula for $\delta$.

Thus when the labelled-$\delta$-chain condition
holds, any $f$ in $\AAA(P)$ has $f-\delta$ in $\AAA^\weak(P)$:
the recursive formula for $\delta$ shows that 
$f-\delta$ is weakly decreasing along each covering relation
$i \lessdot_P j$.

Conversely, if the labelled-$\delta$-condition fails, 
then there exists some covering relation $i \lessdot_P j$ for which
the inequality in \eqref{recursive-delta-definition}
is strict.  In this case, one can check that the
function defined by
$$
f(k) =
\begin{cases}
\delta(k) & \text{ if } k \in P_{\geq i} \text{ but }k \neq j, \\
\delta(k)+1 & \text{ if } k \in P \setminus P_{\geq i} \text{ or }k = j, \\
\end{cases}
$$
gives an element $f$ of $\AAA(P)$ with the property that
$f-\delta$ does not lie in $\AAA^\weak(P)$:  
$$
(f-\delta)(i)=0 > 1=(f-\delta)(j)
$$ 
so $f-\delta$ fails to be weakly order-reversing along the
cover relation $i \lessdot_P j$.
\end{proof}

When $P$ satisfies the labelled-$\delta$-chain condition,
let $\maj(P):=|\delta|=\sum_{i=1}^n \delta(i)$.
The following is then simply a translation of Proposition~\ref{prop:principal-ideal}.

\begin{corollary}
\label{cor:delta-chain-maj}
A poset $P$ on $\{1,2,\ldots,n\}$ satisfies
$$
\sum_{f \in \AAA(P)} \xx^f = \xx^{f_{\min}} \sum_{f \in \AAA^\weak(P)} \xx^f
$$
for some vector $f_{\min}$ in $\NN^n$
if and only if $P$ satisfies the labelled-$\delta$-chain condition.
In this case, $f_{\min} = \delta$, and one has
$$
\sum_{w \in \LLL(P)} q^{\maj(w)} 
 = q^{\maj(P)}  \cdot (1-q)(1-q^2)\cdots (1-q^n) \cdot \Hilb(R_P, q).
$$
\end{corollary}

\begin{example}
\label{maj-forest-formula-example}
Recall from the Introduction that a forest poset $P$ is
one in which an element is covered by at most one other element.
Thus any forest poset $P$ on $\{1,2,\ldots,n\}$ 
always satisfies the labelled-$\delta$-chain condition,
since for each $i$ there is only one maximal chain in $P_{\geq i}$.

Note also that, for forest posets, since no duplications are
used in their construction, $\DDD(P)$ is empty, so that
$\Pi(P)$ is empty, and $\Jconn(P)$ is simply the set of
all principal order ideals $P_{\leq i}$.
Thus one concludes in this case, from 
Theorem~\ref{thm:duplicated-forest-extensions} and 
Corollary~\ref{cor:delta-chain-maj} that
for arbitrarily labelled forest posets $P$
$$
\sum_{w \in \LLL(P)} q^{\maj(w)} 
 = q^{\maj(P)} \,\, \frac{ [n]!_q }{ \prod_{i=1}^n \left[ \,\,\,  |P_{\leq i}| \,\,\, \right]_q }.
$$
This is the major index $q$-hook formula for forests of
Bj\"orner and Wachs \cite[Theorem 1.2]{BjornerWachs}.
See also \cite[\S 6]{BFLR}.
\end{example}

\begin{example}
\label{maj-forest-with-duplication-formula-example}
More generally, there is an easy sufficient (but not necessary)
condition on the labelling of a forest with
duplications $P$ to make it satisfy
the labelled-$\delta$-chain condition: 
for every duplication pair $\{a,a'\}$ and every
duplicated pair of covering edges 
(i.e. either of the form $b \lessdot_P a, a'$ 
or of the form $a, a' \lessdot_P b$), assume that
both covering edges in the pair have the same weak/strict
nature, that is either $b <_\NN a,a'$ or $b >_\NN a,a'$.  

Then for such labellings of a forest with duplications
one has
$$
\sum_{w \in \LLL(P)} q^{\maj(w)} 
 = q^{\maj(P)} \,\,  [n]!_q
 \cdot {\prod_{\{J_1,J_2\} \in \Pi(P)} 
        \left[ \,\,\, |J_1|+|J_2| \,\,\, \right]_q } \Bigg /
               {\prod_{J \in \Jconn(P)} \left[ \,\,\, |J| \,\,\, \right]_q }.
$$
\end{example}

\begin{remark}
Because they are normal affine semigroup rings,
a result of Hochster \cite[Theorem 6.3.5(a)]{BrunsHerzog}
implies that the weak $P$-partition rings $R_P$ are {\it always} 
Cohen-Macaulay.  We have seen that $R_P$ is a complete intersection
whenever $P$ is a forest with duplications, and it will be shown
in the next section that the converse also holds.

Thus one might ask for a combinatorial characterization
of when $R_P$ has the intermediate property of being 
{\it Gorenstein}, that is, the {\it canonical module}
$\Omega(R_P)$ is isomorphic to $R_P$ itself.  This is answered
already by Stanley's work on the $\delta$-chain
condition that was mentioned earlier, as we now explain.  

A result \cite[Theorem 6.3.5(b)]{BrunsHerzog} often
attributed both to Danilov and to Stanley
implies that the canonical module $\Omega(R_P)$ is
isomorphic to the ideal within $R_P$ spanned $k$-linearly by
the monomials $\xx^f$ as $f$ runs through the set $\AAA^\strict(P)$
of all strict $P$-partitions.  Hence $\Omega(R_P)\cong R_P$
exactly when 
$$
\AAA^\strict(P) = f_{\min} + \AAA^\weak(P)
$$
for some $f_{\min}$.  Stanley showed that such an $f_{\min}$
exists (and equals $\delta$) exactly when $P$ satisfies his original
{\it $\delta$-chain condition}, that is, for every $i$, 
all maximal chains in $P_{\geq i}$
have the same length.
\end{remark}

\section{Characterizing complete intersections: 
proof of Theorem~\ref{thm:c.i.-characterization}}
\label{CI-characterization-section}

Recall that in the second proof of Theorem~\ref{thm:duplicated-forest-extensions}
in Section~\ref{second-forest-proof-section}, it was noted
that any of the presentations of three rings
$R_P, \gr(R_P), S/I^\init_P$
given in Theorem~\ref{thm:minimal-presentation}
had the same number of generators and relations.
Thus any of these is a complete intersection presentation if and only
if it is true for all three of them;  we will say that 
$P$ is a {\it c.i. poset} when this holds.  It was further
shown there that forests with duplication $P$ are c.i. posets.
Our goal now is to show that this property
{\it characterizes} forests with duplication.  In the
process, we will encounter more equivalent characterizations,
including one by forbidden induced subposets (Theorem~\ref{ThmForbiddenSubposets}).

\subsection{Nearly principal ideals}

Given a subset $A$ of a elements in a poset $P$,
let $I(A)$ denote the smallest order ideal of $P$ containing $A$,
that is,
$$
I(A):=\{ p \in P: \text{ there exists }a\in A\text{ with }p \leq_P a\}.
$$
Recall also that $\Pi(P)$ denotes the set of pairs
$\{J_1, J_2\}$ of nonempty connected order ideals
of $P$ that intersect nontrivially.

\begin{definition}
Define the set $\BBB(P)$ of all connected, 
{\it nonprincipal} order ideals of $P$,
and define a map 
$$
\begin{array}{rcl}
\Pi(P) & \overset{\pi}{\longrightarrow} &\BBB(P) \\
 \{J_1,J_2\} &\longmapsto &J_1 \cup J_2.
\end{array}
$$
It is easy to check that $\pi$ is well-defined.  It is also surjective:
any nonprincipal connected order ideal $J$ with maximal elements
$j_1,j_2,\ldots,j_m$ for $m \geq 2$ can be written as the union 
$J=J_1 \cup J_2$ where $J_1:=I(j_1)$ and $J_2:=I(j_2,\ldots,j_m)$.

Say that an order ideal $J$ in $\BBB(P)$ is
{\it nearly principal} if its fiber $\pi^{-1}(J)$ for this
surjective map $\pi$ contains only one element.
In other words, $J$ is connected, nonprincipal, and there is a unique
(unordered) pair $\{J_1,J_2\}$ of connected ideals that intersect nontrivially
with union $J_1 \cup J_2=J.$
\end{definition}

It turns out that one can be much more explicit about the
nature of nearly principal ideals;
see Proposition~\ref{nearly-principal-prop}
below.  But our immediate goal is to show how they help characterize
the posets $P$ for which $R_P \cong S/I_P$ is a complete intersection
presentation. 

\begin{proposition}
\label{complete-intersection-prop}
For any poset $P$ on $\{1,2,\ldots,n\}$,
the following are equivalent:
\begin{enumerate}
\item[(i)] Any or all of the presentations $R_P \cong S/I_P$
and  $\gr(R_P) \cong S/I^\gr_P$ and $S/I^\init_P$
are complete intersection presentations.
\item[(ii)] $|\Pi(P)|=|\BBB(P)|=|\Jconn(P)|-|P|$.
\item[(iii)] The surjection $\pi: \Pi(P) \rightarrow \BBB(P)$ is a bijection.
\item[(iv)] Every connected order ideal of $P$
is either principal or nearly principal.
\end{enumerate}
\end{proposition}
\begin{proof}
The equivalence of (i) and (ii) appeared already in the second proof of 
Theorem~\ref{thm:duplicated-forest-extensions}.
The equivalence between (ii) and (iii) is trivial, since by
definition one has the equality
$
|\BBB(P)|=|\Jconn(P)|-|P|.
$ 
The equivalence of (iii) and (iv) is immediate from the
definition of a nearly principal ideal.
\end{proof}

Say that $Q$ is an {\it (induced) subposet} of $P$
if one has an injective map $i: Q \rightarrow P$ for which 
$i(q) \leq_P i(q')$ if and only if $q \leq_Q q'$.
Condition (iv) of Proposition~\ref{complete-intersection-prop}
lets one deduce the following.

\begin{corollary}
\label{PropCIStab}
Induced subposets of c.i.-posets are c.i.-posets.
\end{corollary}
\begin{proof}
Given an injective map $i:Q \rightarrow P$ as above,
and an order ideal $J$ of $Q$ which is connected 
(resp. principal, resp. nearly principal), 
one readily checks that the order ideal $I(i(J))$ of $P$ is 
connected (resp. principal, nearly principal).  
Thus if the subposet $Q$ is not c.i.,
then it contains a connected order ideal $J$ 
which is neither principal nor nearly principal
by Proposition~\ref{complete-intersection-prop}(iv),
and then $P$ contains the connected order ideal 
$I(i(J))$ which is neither principal nor nearly principal,
so that $P$ is also not c.i.
\end{proof}

\noindent
Corollary~\ref{PropCIStab} implies that c.i.-posets are 
exactly the posets avoiding some family of ``forbidden''
posets as induced subposets.
This forbidden family might, {\it a priori}, be 
infinite\footnote{For example, consider the 
family of {\it crown posets} $\{C_n\}_{n \geq _2}$, where $C_n$
has $2n$ elements $\{a_1,\ldots,a_n,b_1,\ldots,b_n\}$ and relations
$$
a_1<b_1>a_2<b_2>a_3<b_3> \cdots <b_{n-2}>a_{n-1}<b_{n_1}>a_n<b_n>a_1.
$$
No two crowns $C_i, C_j$ for $i \neq j$ contains one another as an 
induced subposet, so the family of posets avoiding crowns as 
induced posets is not characterized by avoiding some finite subfamily.}.
Our next goal is to show that c.i. posets
are characterized by avoiding the three posets $P_1, P_2, P_3$
shown in Theorem~\ref{ThmForbiddenSubposets} below. 
For this, it helps to start with a more explicit description 
of nearly principal order ideals. 

\begin{proposition}
\label{nearly-principal-prop}
A connected nonprincipal order ideal $J$ of a finite poset $P$ is
nearly principal if and only if 
\begin{enumerate}
\item[(a)] it has exactly two maximal elements $j_1, j_2$, and
\item[(b)] for every common lower bound $\ell <_P j_1, j_2$, 
the open intervals $]\ell,j_1[$ and $]\ell,j_2[$ coincide.
\end{enumerate}
\end{proposition}

\begin{proof}
For the ``only if'' assertion, let $J$ be a connected nonprincipal
order ideal in $P$ that fails one of the two conditions above.
\begin{itemize}
  \item If $J$ fails (a), having distinct maximal elements 
  $j_{1},j_{2},\cdots,j_m$ with $m \geq 3$,
  then it can be written in at least two ways as a union of connected order
  ideals intersecting nontrivially:
     $$
     \begin{aligned}
      J &= I(j_1) \cup I(j_2,j_3,j_4,\ldots,j_{m}) \\
        &= I(j_2) \cup I(j_1,j_3j_4,\ldots,j_{m})
     \end{aligned}
     $$
  Hence $J$ is not nearly principal.
  \item If $J$ satisfies (a), so that it
  has two maximal elements $j_1$ and $j_2$,
  but fails (b) by having a lower bound $\ell <_P j_1, j_2$
  and an element $k$ of $]\ell,j_1[$ not lying in $]\ell,j_2[$, 
  then $J$ can again be written in at least two ways as a union of 
  connected order ideals intersecting nontrivially
   $$
   \begin{aligned}
       J &= I(j_{1}) \cup I(j_{2}) \\
         &= I(j_{1}) \cup I(j_{2},k).
   \end{aligned}
   $$
  Note that $I(j_2,k)$ is connected because it is the 
  union of two principal ideals that both contain $\ell$.
  This shows $J$ is not nearly principal. 

\end{itemize}

For the ``if'' assertion, assume that $J$ is a connected nonprincipal
ideal satisfying conditions (a), (b) above.  We wish to show that,
given any expression $J = J_1 \cup J_2$ where $J_1, J_2$ are connected
order ideals intersecting nontrivially, one can re-index
so that $J_1=I(j_1)$ and $J_2=I(j_2)$.  By condition (a), one
can re-index without loss of generality so that 
$j_1 \in J_1 \setminus J_2$ and $j_2 \in J_2 \setminus J_1$.
Therefore $I(j_1) \subseteq J_1$, so it only remains to show
the reverse inclusion, that is, $J_1 \setminus I(j_1)$ is empty.  
If not, then by the connectivity of $J_1$, there must exist 
$k, \ell$ with $k \in J_1 \setminus I(j_1)$
and $\ell \in I(j_1)$ such that $k,\ell$ are comparable in $P$.

If $k <_P \ell$, then together with $\ell \leq_P j_1$,
this contradicts $k \not\in I(j_1)$.

If $\ell <_P k$, then note that $k \in J=I(j_1,j_2)$ together with
$k \not\in I(j_1)$ forces $k \leq j_2$.  Thus $\ell <_P k \leq j_2$
so that $\ell$ is a lower bound for $j_1, j_2$.  However,
then $k$ lies in $]\ell,j_2[$ but not in $]\ell,j_2[$, 
contradicting condition (b).
\end{proof}

\subsection{Two further characterizations of c.i. posets}

\begin{theorem}\label{ThmForbiddenSubposets}
The c.i. posets are those which do not contain
any of the following three posets $\{P_1,P_2,P_3\}$
as induced subposets:

\vskip.1in
\qquad\qquad
\epsfxsize=90mm
\epsfbox{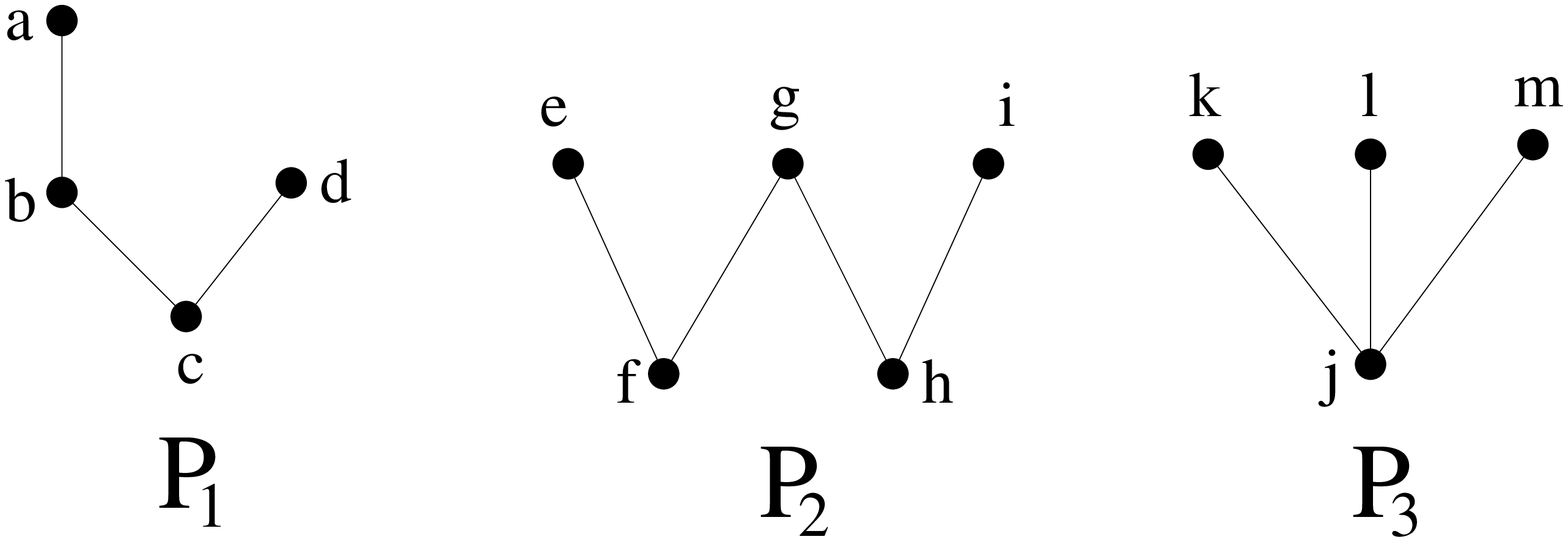}

\end{theorem}

\begin{proof}
Each of $P_{1}$, $P_{2}, P_{3}$ is not a c.i.-poset
because it is itself an order ideal $J=P_i$ 
which is connected but neither principal nor nearly principal.  For example,
one can exhibit these two different decompositions into connected 
order ideals intersecting nontrivially:
$$
\begin{array}{rcccl}
P_1 &=&I(a) \cup I(d)
&=&I(a) \cup I(b,d) \\
P_2 &=&I(e) \cup I(g,i) 
&=&I(e,g) \cup I(i) \\
P_3 &=&I(k) \cup I(\ell,m)
&=&I(k,\ell) \cup I(m) \\
\end{array}
$$

By Proposition~\ref{complete-intersection-prop}(iv) and
it only remains to show that, if a poset $P$ contains a connected
nonprincipal order ideal $J$ failing either of
the conditions (a), (b) in Proposition~\ref{nearly-principal-prop}, 
then $P$ contains one of $P_1,P_2,P_3$ as induced subposets.

First assume $J$ fails condition (a), having distinct maximal elements
$j_1,j_2,\ldots,j_m$ with $m \geq 3$.  Then connectivity of $J$ forces
$I(j_1) \cap I(j_2,j_3,\ldots,j_m)$ to contain at least one element,
whom we will denote $\ell_1$, and re-index so that $\ell_1 \leq_P j_1, j_2$.
Again, connectivity of $J$ forces 
$I(j_1,j_2) \cap I(j_3,j_4,\ldots,j_m)$ to contain at least one element,
whom we will denote $\ell_2$, and without loss of generality, one can
again re-index so that $\ell_2 \leq_P j_2, j_3$.  Now there are three cases:
\begin{enumerate}
   \item[$\bullet$]
   if $\ell_1 \leq_P j_3$ (which holds in particular if $\ell_1 \leq_P \ell_2$),
   then $\{j_1,j_2,j_3,\ell_1\}$ induces a subposet of $P$ isomorphic to $P_3$;
   \item[$\bullet$]
   in a symmetric way, if $\ell_2 \leq_P j_1$
   (which holds in particular if $\ell_2 \leq_P \ell_1$),
   then $\{j_1,j_2,j_3,\ell_2\}$ induces a subposet of $P$ isomorphic to $P_3$;
\item[$\bullet$]
otherwise, $\ell_1,\ell_2$ are incomparable in $P$ and
$\{j_1,j_2,j_3,\ell_1,\ell_2\}$
induces a subposet of $P$ isomorphic to $P_2$.
\end{enumerate}

Finally, assume that $J$ satisfies condition (a), so that
$J=I(j_1,j_2)$, but $J$ fails condition (b), due to the existence of a 
lower bound $\ell <_P j_1,j_2$ and 
(without loss of generality by re-indexing) some element $k$ in
$]\ell,j_1[$ but not in $]\ell,j_2[$.  Then $\{j_1,j_2,k,\ell\}$
induce a subposet of $P$ isomorphic to $P_1$.
\end{proof}

\begin{theorem}\label{ThmEqFwdCip}
The set of c.i. posets is exactly the set of forests with duplications.
\end{theorem}

\begin{proof}
It was already been proven in Section~\ref{second-forest-proof-section} 
that a forest with duplication is a c.i. poset.  Conversely, given
a c.i. poset $P$, we will show by induction on $|P|$ that
it is a forest with duplications.

The base case $|P|=1$ is trivial. 
In the inductive step, if $P$ contains no two comparable
elements, then $P$ is a disjoint union of posets with one
element, and hence a forest with duplication.
Otherwise, let $a$ be a non minimal element of $P$, and
we consider two cases.

\vskip.1in
\noindent
{\sf Case 1:} Every element $a'$ incomparable to $a$ in $P$ has
$I(a') \cap I(a) = \varnothing$.  

In this case, consider the (nonempty) induced 
subposets $P_{<a}$ and $P \setminus P_{<a}$ in $P$.  
Both are c.i. posets by Proposition~\ref{PropCIStab},
and both have fewer elements than $P$, so they are forests with 
duplication by induction.  And it is straightforward
to check that, in this situation,
$P$ is isomorphic to the poset obtained by
hanging $P_{<a}$ below $a$ in $P \setminus P_{<a}$.
Therefore $P$ is also a forest with duplication.

\vskip.1in
\noindent
{\sf Case 2:} There exists an element $a'$ incomparable to $a$ in $P$ 
for which $I(a') \cap I(a) \neq \varnothing$.

In this case, decompose $P$ into four induced subposets
\begin{equation}
\label{four-pieces}
P=\hat{P} \quad \sqcup \quad P_{<a,a'} 
          \quad \sqcup \quad  P_{<a} \setminus P_{<a,a'} 
          \quad \sqcup \quad P_{<a'} \setminus P_{<a,a'} 
\end{equation}
where $\hat{P}:=P\setminus \left( P_{<a} \cup P_{<a'} \right)$,
and where $P_{<a} \setminus P_{<a'}$ and
$P_{<a'} \setminus P_{<a}$ are allowed to be empty,
but $P_{<a,a'}$ is not.  This decomposition is depicted schematically here:

\vskip.1in
\qquad\qquad
\epsfxsize=90mm
\epsfbox{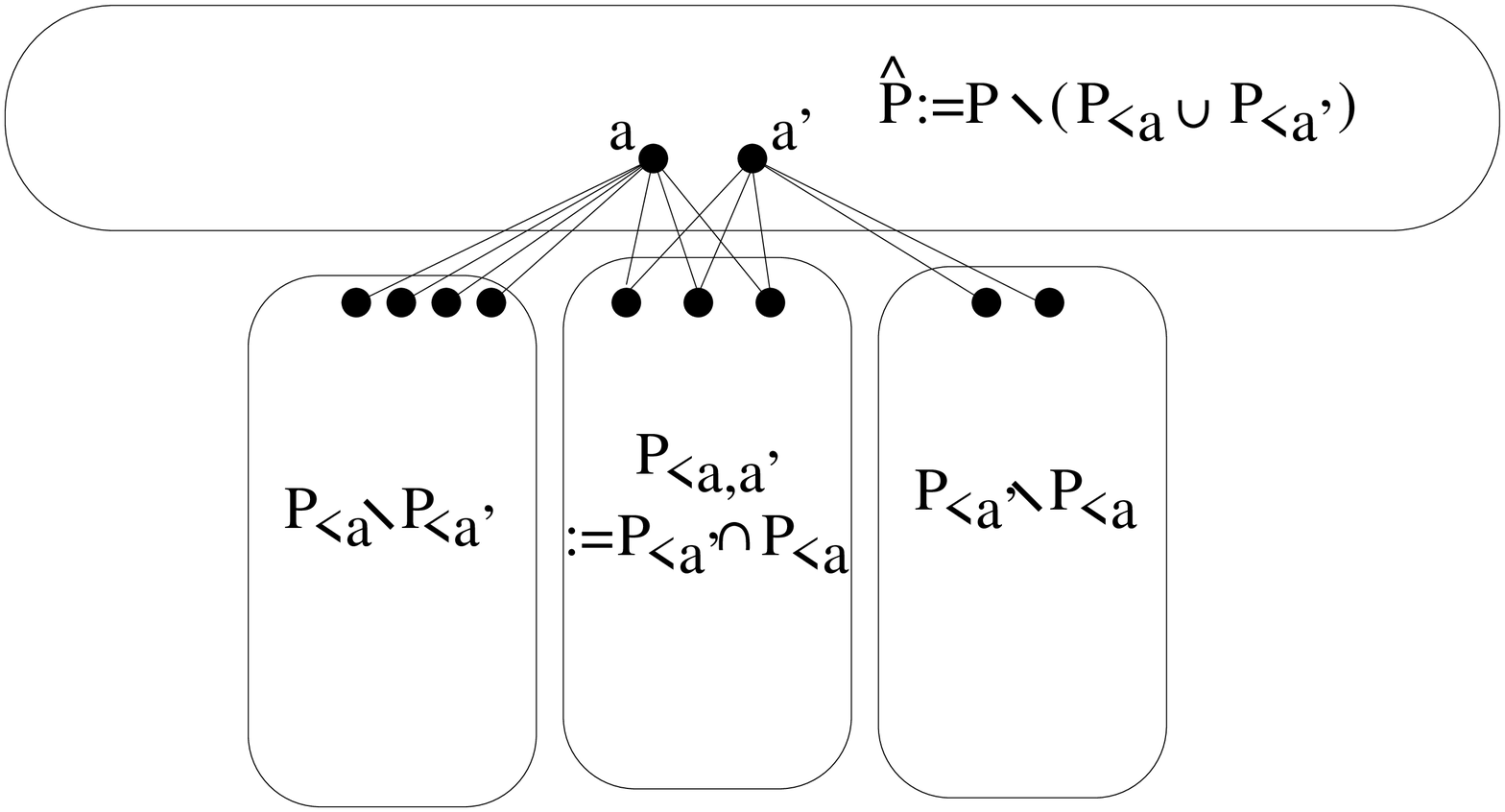}

\noindent
We will show that $P$ is isomorphic to the poset $Q$ built by this  process:
\begin{itemize}
    \item[(1)] Start with $\hat{P} \setminus \{a'\}$.
    \item[(2)] Hang $P_{<a,a'}$ below $a$ in $\hat{P} \setminus \{a'\}$.
    \item[(3)] Duplicate the hanger $a$ in the result, with duplicate 
           element denoted $a'$.
    \item[(4)] Hang $P_{<a} \setminus P_{<a'}$ (if it is nonempty) below $a$, and
         \newline hang $P_{<a'} \setminus P_{<a}$ (if it is nonempty) 
                   below $a'$ in the resulting poset.
\end{itemize}
Since $\hat{P} \setminus \{a'\}$ and $P_{<a,a'}$,
and $P_{<a} \setminus P_{<a'}$ and $P_{<a'} \setminus P_{<a}$ 
are all induced subposets of $P$, they are all c.i. posets 
by Proposition~\ref{PropCIStab}.  Since they have
smaller cardinality than $P$, they are all forests with duplication
by induction.  Therefore $Q$ is also a forest with duplication.

It only remains to show that $P$ is isomorphic to $Q$.
Their underlying sets are the same.  It should also be clear that,
by construction, $P$ and $Q$ have the same restrictions to the last three
pieces on the right side of \eqref{four-pieces}.  For the first piece $\hat{P}$
this is also true, for the following reason:  since $P_{<a,a'}$
is assumed to contain at least one element $\ell$, 
any element $b$ of $P$ will have $b>_P a$ if and only $b >_P a'$, 
else $\{b,a,a',\ell\}$ would induce a subposet
of $P$ isomorphic to $P_1$. 

Now given two elements $x,y$ lying in two {\it different} pieces
from the decomposition \eqref{four-pieces}, 
one must check that $x,y$ are related the same way in $P$ and $Q$.
This is checked case-by-case, according to the two pieces
in which they lie.

\vskip.1in
\noindent
{\sf 
$x$ lies in $P_{<a} \setminus P_{<a,'}$
and $y$ lies in $P_{<a'} \setminus P_{<a}$.}
\newline
Here transitivity implies that
$x,y$ are incomparable both in $P$ and in $Q$.

\vskip.1in
\noindent
{\sf 
$x$ lies in $P_{<a} \setminus P_{<a,'}$ or $P_{<a'} \setminus P_{<a}$
and $y$ lies in $P_{<a,a'}$.}
\newline
Then $x,y$ are incomparable in $Q$.
But the same holds in $P$: if $x <_P y$ then it would contradict 
$x \not\in P_{<a,a,'}$ by transitivity, and if $y <_P x$ then $\{a,a',x,y\}$ 
induces a subposet of $P$ isomorphic to $P_1$.

\vskip.1in
\noindent
{\sf 
$x$ lies in $P_{<a} \setminus P_{<a,'}$ and $y$ lies in $\hat{P}$.}
\newline
Then $y \leq_Q x$ and $y \leq_P x$ are both impossible by transitivity.  
Thus one must check that $x \leq_Q y$ if and only if $x \leq_P y$.
One has $x \leq_Q y$ if and only if $a \leq_P y$, and it is true
that $a \leq_P y$ implies $x \leq_P y$ by transitivity.
Thus it remains to check the converse:  
$a \not\leq_P y$ implies $x \not\leq_P y$.
Assuming $a \not\leq_P y$, if one had $x \leq_P y$, then
pick $\ell$ to be any element of the nonempty subset $P_{<a,a'}$.
Either $\ell \not\leq y$ and
$\{y,x,a,a',\ell\}$ induces a subposet of $P$ isomorphic to $P_2$,
or $\ell \leq y$ and $\{\ell,y,a,a'\}$ induces a subposet
isomorphic to $P_3$.  Contradiction.

\vskip.1in
\noindent
{\sf 
$x$ lies in $P_{<a'} \setminus P_{<a}$ and $y$ lies in $\hat{P}$.}
\newline
Swapping the roles of $a,a'$ puts one in the case just
considered.

\vskip.1in
\noindent
{\sf 
$x$ lies in $P_{<a,a'}$ and $y$ lies in $\hat{P}$.}
\newline
Again $y \leq_Q x$ and $y \leq_P x$ are both impossible by transitivity.  
Thus one must check that $x \leq_Q y$ if and only if $x \leq_P y$.
One has $x \leq_Q y$ if and only if either $a \leq_P y$ or $a' \leq_P y$.
Furthermore, either $a \leq_P y$ or $a' \leq_P y$ will imply
$x \leq_P y$ by transitivity.  Thus it remains to check the converse:  
if both $a \not\leq_P y$ and $a' \not\leq_P y$ then this forces
$x \not\leq_P y$.  This follows since otherwise if $x \leq_P y$
then $\{y,a,a',x\}$ induces a subposet isomorphic to $P_3$ in $P$.

\vskip.1in
\noindent
This completes the proof that $P$ is isomorphic to the
forest with duplication $Q$.
\end{proof}

\section{Geometry of $I^\init_P$, graph-associahedra
and graphic zonotopes}
\label{triangulation-section}

  Our goal in this section is to explain the geometry underlying
Proposition~\ref{P-partition-expressions-prop}(ii) and the
initial ideal $I^\init_P$, in terms of a subdivision 
of the cone of $P$-partitions.  We explain how
\begin{enumerate}
\item[$\bullet$]
the cone of $P$-partitions is the normal cone $\NNN_\omega$
at a particular vertex $\omega$ in the graphic zonotope
$\ZZZ_{G}$ associated to the Hasse diagram graph $G$ of $P$, 
\item[$\bullet$]
the normal fan of $\ZZZ_{G}$ is refined by the (simplicial) normal fan
of Carr and Devadoss's {\it graph-associahedron} $\PPP_{\BBB(G)}$ associated
to $G$, and
\item[$\bullet$]
the initial ideal $I^\init_P$ is exactly
the Stanley-Reisner ideal $I_{\Delta(P)}$
for the simplicial complex $\Delta_P$
describing the triangulation of the cone $\NNN_\omega$ 
by the normal fan of $\PPP_{\BBB(G)}$.
\end{enumerate}

\begin{definition}
Let $\Delta_P$ denote the simplicial complex 
having the squarefree monomial ideal 
$I^\init_P$ in the polynomial algebra $S=k[U_J]_{J \in \Jconn(P)}$
as its {\it Stanley-Reisner ideal} $I_{\Delta_P}$.  By definition
this means that $\Delta_P$ is the abstract simplicial complex
with vertex set indexed by the collection
$\Jconn(P)$ of nonempty connected order ideals $J$ in $P$,
and a subset $\{J_1,\ldots,J_d\}$ forms a $(d-1)$-simplex
of $\Delta_P$ if and only if the $\{J_i\}$ pairwise intersect
trivially (either disjointly, or nested).
\end{definition}

Recall that a {\it flag (or clique) complex} is an abstract
simplicial complex $\Delta$ on a vertex set $V$ having the following
property:  whenever a subset $\sigma \subset V$ has every pair 
$\{i,j\} \subset \sigma$ spanning an edge of $\Delta$, then the entire
subset $\sigma$ spans a simplex of $\Delta$.

We refer the reader to Stanley \cite[\S III.2 and III.10]{Stanley-CCA}
for the notions of {\it shellability} and {\it regular triangulations}
used in the next result.

\begin{proposition}
For any poset $P$ on $\{1,2,\ldots,n\}$, the simplicial complex
$\Delta_P$ is a flag simplicial complex,
giving a regular triangulation of 
a shellable $(n-2)$-dimensional ball.
\end{proposition}
\begin{proof}
The fact that $\Delta_P$ is flag comes from the fact that
$I^\init_P$ is generated by (squarefree) {\it quadratic} monomials.
The fact that it gives a regular (and hence shellable) triangulation of a ball
comes from a general result of Sturmfels on
initial ideals and regular triangulations; see 
\cite[Chapter 8]{Sturmfels}.
\end{proof}

We wish to relate $\Delta_P$ to the {\it normal fans} of two polytopes 
associated to the (undirected) graph $G$ on vertex
set $\{1,2,\ldots,n\}$ which is the Hasse diagram of $P$:
\begin{enumerate}
\item[$\bullet$]
the {\it graphic zonotope}
$\ZZZ_G$, and 
\item[$\bullet$]
the {\it graph-associahedron} $\PPP_{\BBB(G)}$ of Carr and Devadoss \cite{CarrDevadoss}.
\end{enumerate}
For a discussion of polytopes, normal fans, and zonotopes, 
see Ziegler's book \cite[Chapter 7]{Ziegler};
for graphic zonotopes and graph-associahedron, 
see \cite[\S 5-7]{PostnikovRWilliams}.

Recall that for two subsets $A,B \subset \RR^n$, their {\it Minkowski sum} is
$$
A+B = \{a+b: a \in A, b \in B\}.
$$ 

\begin{definition}
  The {\it graphic zonotope} $\ZZZ_G$ is the Minkowski sum of the line segments
$\{ [0,e_i-e_j] \}_{\{i,j\} \in E}$.  In particular, taking $G=K_n$,
one has that $Z_{K_n}$ is the $n$-dimensional {\it permutohedron}.
\end{definition}

\begin{definition}
  The {\it graphical building set} $\BBB(G)$ is the collection of
all nonempty vertex subsets $J \subseteq \{1,2,\ldots,n\}$ for which the
vertex-induced subgraph $G|_J$ is connected.

  The {\it graph-associahedron} $\PPP_{\BBB(G)}$ is the Minkowski sum
of the simplices 
$$
\{ \conv(\{e_j\}_{j \in J}) : J \in \PPP_{\BBB(G)}\}
$$
where here $\conv(A)$ denotes the convex hull of the vectors in $A$. 
\end{definition}

Recall that for a convex polytope $\PPP$ in $V=\RR^n$, its {\it normal
fan} $\NNN(\PPP)$ is the collection of cones in the dual
space $V^*$ which partitions linear functionals according
to the face of $\PPP$ on which they achieve their maximum value.  
We will use repeatedly the following well-known fact about
normal fans of Minkowski sums.

\begin{proposition}(see e.g. Ziegler \cite[Prop. 7.12]{Ziegler})
\label{prop:Minkowski-fans}
\newline
The Minkowski sum 
$
\PPP_1 + \cdots + \PPP_d
$ 
has normal fan
$\NNN(\PPP_1 + \cdots + \PPP_d)$ equal to the common refinement of
the normal fans $\NNN(\PPP_1),\ldots,\NNN(\PPP_d)$.
\end{proposition}

\begin{proposition}
\label{prop:Minkowski-sum} 
Let $G$ be a graph on vertex set $\{1,2,\ldots,n\}$.
\begin{enumerate}
\item[(i)]
The normal fan $\NNN(\ZZZ_G)$ is the collection of
cones in $\RR^n$ cut out by the graphic arrangement
of hyperplanes $\{ x_i=x_j\}_{\{i,j\} \in E}$.
\item[(ii)]
In particular, when $G$ is the complete graph $K_n$, this graphic arrangement
is the usual type $A_{n-1}$ braid or Weyl chamber arrangement.
\item[(iii)]
The braid arrangement $\NNN(\ZZZ_{K_n})$ 
refines the normal fan $\NNN(\PPP_{\BBB(G)})$.
\item[(iv)] 
The normal fan $\NNN(\PPP_{\BBB(G)})$
in turn refines the normal fan $\NNN(\ZZZ_G)$.
\end{enumerate}
\end{proposition}
\begin{proof}
Assertion (i) is well-known, and
follows from the fact that
the hyperplane $x_i=x_j$ is normal
to the line segment $[0,e_i-e_j]$; see e.g.
\cite[\S 5]{PostnikovRWilliams}.

Assertion (ii) is simply a definition of the
type $A_{n-1}$ braid arrangement, as  the collection
of all hyperplanes $x_i=x_j$ for $1 \leq i < j \leq n$.

Assertion (iii) is asserting another well-known fact:
that $\PPP_{\BBB(G)}$ is a {\it generalized permutohedron}
in the sense of Postnikov \cite{Postnikov}; see
\cite[Example 6.2]{PostnikovRWilliams}.  This follows
from Proposition~\ref{prop:Minkowski-fans} 
by checking that each simplex $\conv(\{e_j\}_{j \in J})$ 
has its normal fan refined by the
braid arrangement.  The latter holds because a typical edge
of $\conv(\{e_j\}_{j \in J})$ between vertex 
$e_i$ and vertex $e_j$ is normal to the hyperplane $x_i=x_j$.

Assertion (iv) follows from Proposition~\ref{prop:Minkowski-sum} 
by noting that for each edge $\{i,j\}$ of $G$,
the normal hyperplane $x_i=x_j$ to the Minkowski
summand $[0,e_i-e_j]$ of $\ZZZ_G$ is the normal hyperplane to the
Minkowski summand $\conv(\{e_i,e_j\})$
of $\PPP_{\BBB(G)}$.
\end{proof}

We next review basic facts about the structure of the normal fans
for the permutohedron $\ZZZ_{K_n}$, 
graphic zonotope $\ZZZ_G$, and
graph associahedron $\PPP_{\BBB(G)}$, all inside $\RR^n$.

\vskip.1in
\noindent
{\sf Permutohedron.}
Rays in the normal fan $\NNN(\ZZZ_{K_n})$
are indexed by nonempty proper subsets $J$ of $\{1,2,\ldots,n\}$;
such a ray is the nonnegative span of the characteristic vector
$\chi_J$ in $\RR^n$.  The maximal cones are indexed
by permutations $w=(w_1,\ldots,w_n)$ and defined 
by the inequalities $x_{w_1} \geq x_{w_2} \geq \cdots \geq x_{w_n}.$  
A ray indexed by a subset $J$ lies in the cone indexed by $w$ 
if and only if $J=w|_{[1,i]}$ for some $i=1,2,\ldots,n-1$.

\vskip.1in
\noindent
{\sf Graphic zonotope.}
Maximal cones in the normal fan $\NNN(\ZZZ_{G})$,
or vertices in the graphic zonotope, are indexed by 
{\it acyclic orientations} $\omega$ of the
graph $G$;  such a cone corresponds to the subset of $\RR^n$ defined by
the conjunction of the inequalities $x_i \geq x_j$ whenever 
$\omega$ directs an edge of $G$ as $i \rightarrow j$.  
In slightly different terms, the transitive closure
of an acyclic orientation $\omega$ gives a partial
order $P_\omega$ on $\{1,2,\ldots,n\}$, and
the maximal cone $\NNN_\omega$ of  $\NNN(\ZZZ_{G})$ corresponding to $\omega$
is the cone of (weak) $P_\omega$-partitions.  The decomposition of
Proposition~\ref{fundamental-P-partition-proposition} comes 
from expressing this cone $\NNN_\omega$ as the union of the maximal 
cones of $\NNN(\ZZZ_{K_n})$ corresponding to permutations
$w$ in the set of linear extensions $\LLL(P_\omega)$.

\vskip.1in
\noindent
{\sf Graph associahedron.}
Rays in the normal fan $\NNN(\PPP_{\BBB(G)})$
are a subset of the rays in $\NNN(\ZZZ_{K_n})$:
one only includes the rays indexed by nonempty proper 
subsets $J$ of $\{1,2,\ldots,n\}$ for which the 
vertex-induced subgraph $G|_J$ is {\it connected}.  In other
words, $J$ is required to be an element of the {\it graphical building set} $\BBB(G)$.
A collection of rays $\{J_1,\ldots,J_t\}$ spans a cone in  $\NNN(\PPP_{\BBB(G)})$
if and only if pairwise one has that $J_i, J_k$ intersect trivially
(either they are disjoint or nested) and if disjoint,
then $J_i \cup J_k$ induces a disconnected subgraph $G|_{J_1 \cup J_2}$
(that is, $J_1 \cup J_2$ is not in $\BBB(G)$).
Such collections form the simplices in what is
called the {\it nested set complex} $\Delta_{\BBB(G)}$ for the building
set $\BBB(G)$.

\begin{figure}
$$
\NNN(\ZZZ_G)
$$
$$
\epsfxsize=40mm
\epsfbox{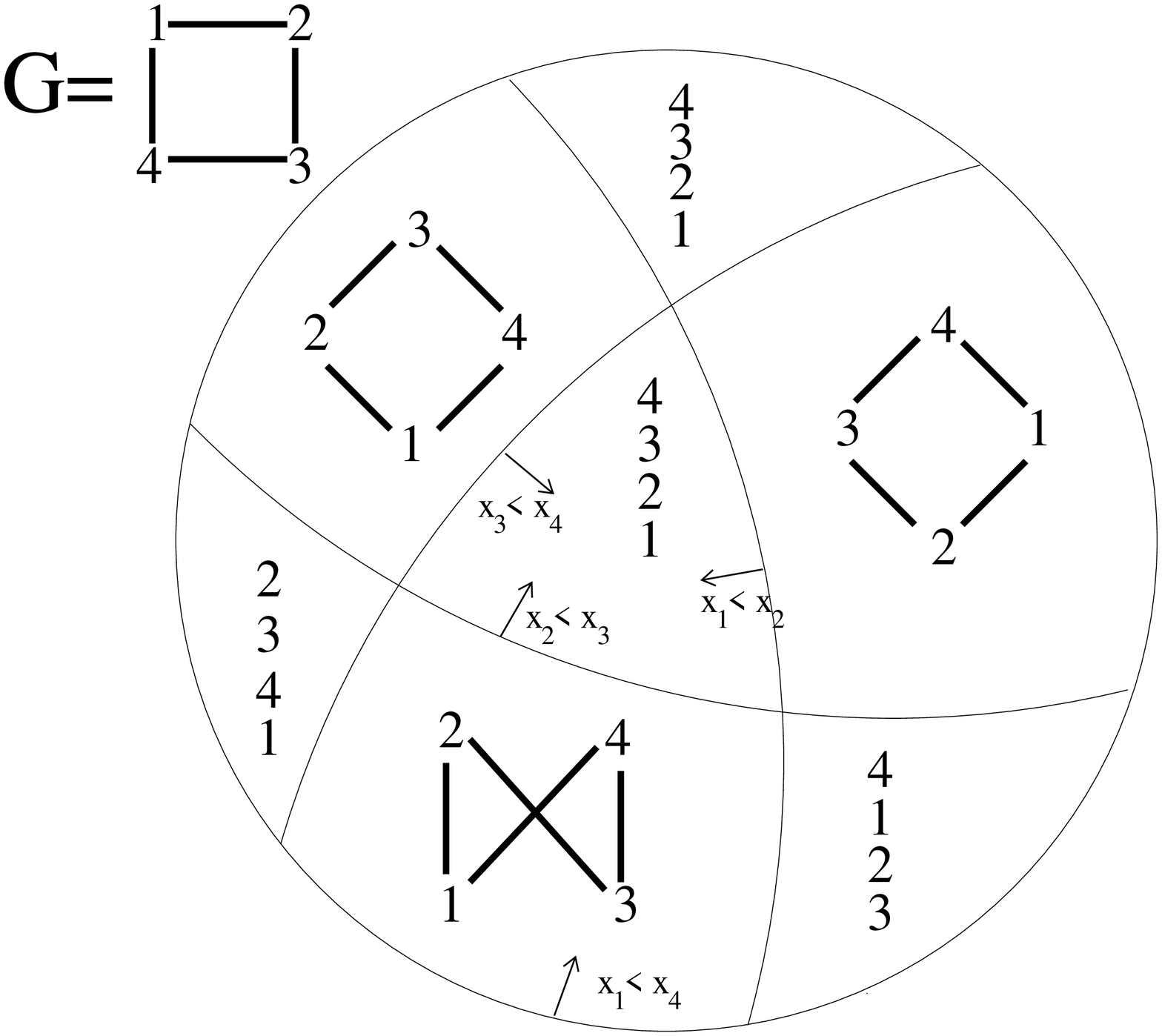}
$$
$$
\NNN(\ZZZ_{K_4})
\epsfxsize=40mm
\epsfbox{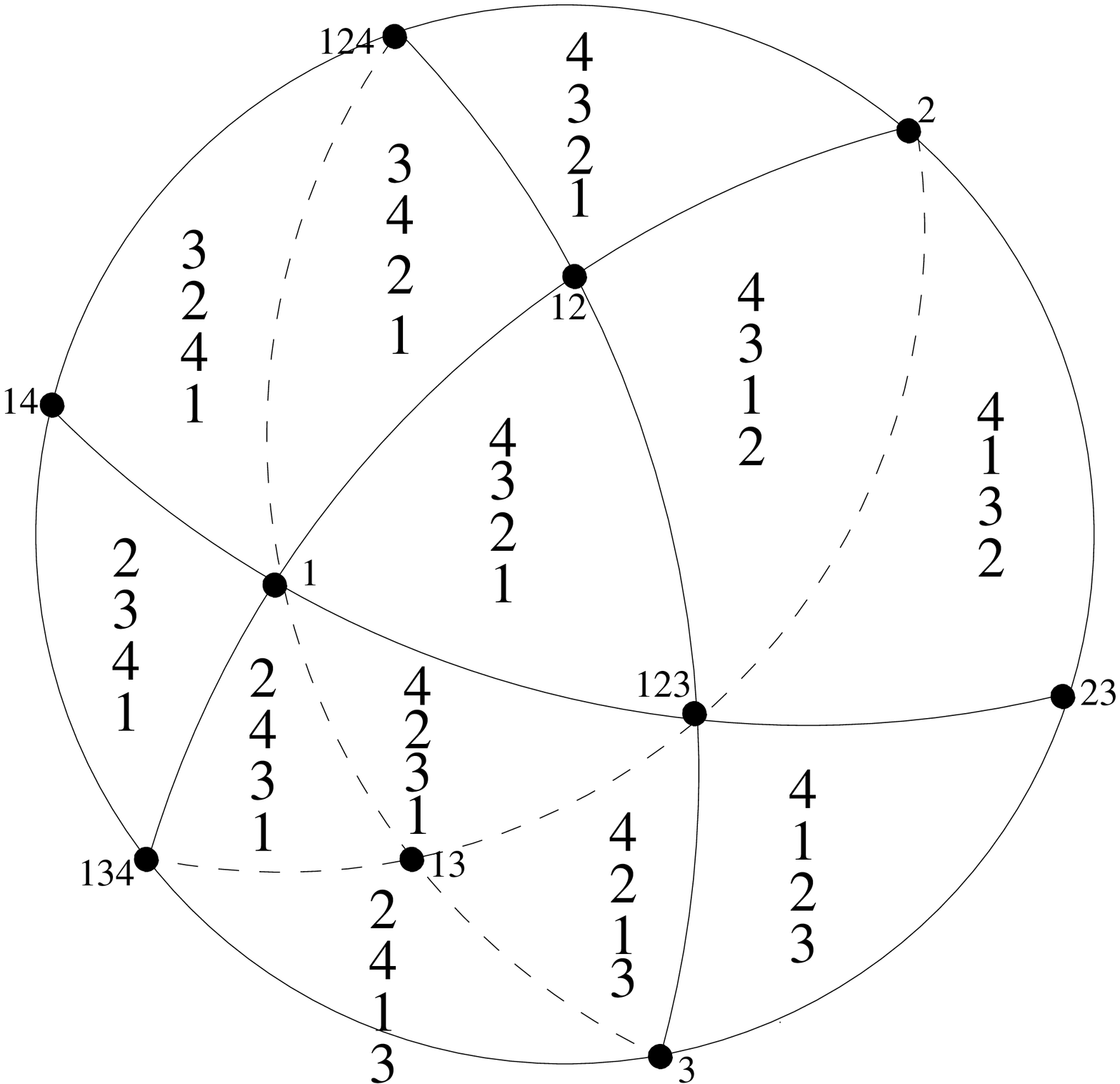}
\quad 
\NNN(\PPP_{\BBB(G)})
\epsfxsize=40mm
\epsfbox{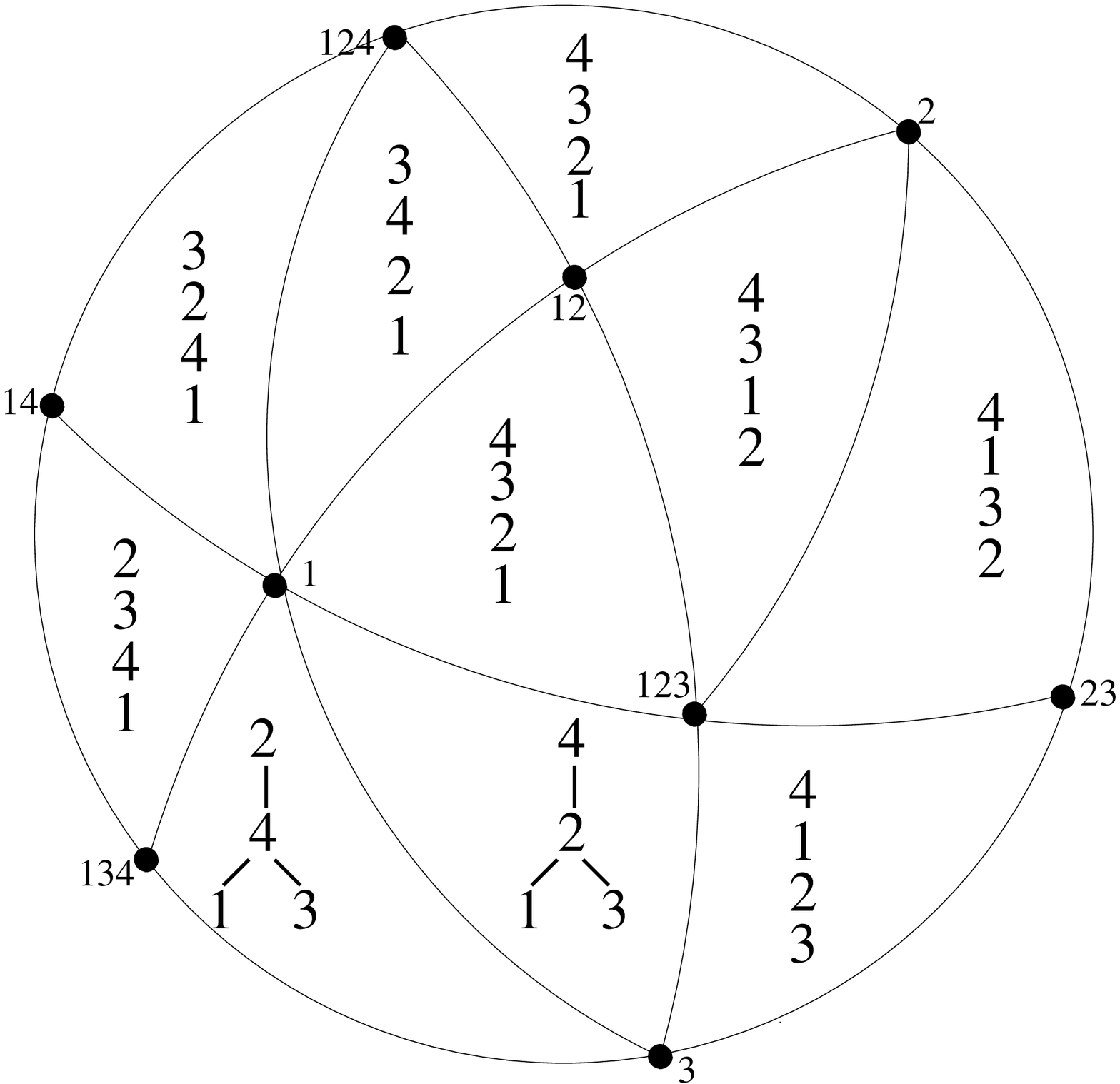}
$$
\caption{Normal fans for 
the graphic zonotope $\ZZZ_G$, 
the permutohedron $\ZZZ_{K_n}$, and
the graph associahedron $\PPP_{\BBB(G)}$, for the graph $G$ shown,
having $n=4$ vertices.
The normal fans live in $\RR^4$, but are depicted inside the hyperplane
$x_1+x_2+x_3+x_4=0$ via their intersection with the hemisphere
of the unit sphere in which $x_1 \geq x_4$.  Note that 
$\NNN(\ZZZ_{K_n})$ refines $\NNN(\PPP_{\BBB(G)})$, and the latter
refines $\NNN(\ZZZ_{G})$.}
\label{normal-fan-figure}
\end{figure}

\begin{proposition}
Given a poset $P$ on $\{1,2,\ldots,n\}$,
with Hasse diagram $G$, let $\omega$ be
the acyclic orientation having $P$ as its
transitive closure.

Then the simplicial complex $\Delta_p$ having
$I_{\Delta_P} = I^\init_P$ describes the triangulation
of the $P$-partition maximal cone $\NNN_\omega$ 
in the fan $\NNN(\ZZZ_G)$ by cones of 
the normal fan $\NNN(\PPP_{\BBB(G)})$.  
\end{proposition}
\begin{proof}
Temporarily let $\Gamma_P$ denote the simplicial
complex describing the triangulation of $\NNN_\omega$ 
in the fan $\NNN(\ZZZ_G)$ by cones of the normal fan 
$\NNN(\PPP_{\BBB(G)})$.  We wish to show that $\Delta_P \cong \Gamma_P$.
As a preliminary reduction, assume that $P$ is connected: when
$P$ is a disjoint union $P_1 \sqcup P_2$ of two other posets,
one can check that 
$$
\begin{aligned}
\Delta_{P} &\cong \Delta_{P_1} * \Delta_{P_2} \\
\Gamma_{P} &\cong \Gamma_{P_1} * \Gamma_{P_2} 
\end{aligned}
$$
where here $*$ denotes the {\it simplicial join}
operation; cf. \cite[Remark 6.7]{PostnikovRWilliams}.

Since $\Delta_P$ is a flag complex, it suffices to check
that $\Gamma_P$ is also a flag complex, and that their
$1$-skeleta (=vertices and edges) are isomorphic.

Recall that $\Delta_P$ has vertex set given by
the set $\Jconn(P)$ of connected order ideals $J$ in $P$, 
with two vertices $\{J_1, J_2\}$ spanning an edge of $\Delta_P$
if and only if the order ideals $J_1, J_2$ intersect trivially
(either disjoint, or nested).

On the other hand, $\Gamma_P$ is the 
subcomplex of the nested set complex $\Delta_{\BBB(G)}$
indexing the cones of $\NNN(\PPP_{\BBB(G)})$
that lie in the cone $\NNN_\omega$.  Note that a
cone lies in $\NNN_\omega$ if and only if each of its
extreme rays lies in $\NNN_\omega$.  Thus $\Gamma_P$ is a
vertex-induced subcomplex of the flag complex
$\Delta_{\BBB(G)}$, and hence is itself flag.

Vertices of $\Delta_{\BBB(G)}$ are
indexed by nonempty proper subsets $J$ of $\{1,2,\ldots,n\}$
for which $G|_J$ is connected.  The extra condition that
$J$ indexes a ray inside $\NNN_\omega$ is equivalent
to $\chi_J$ being a weak $P$-partition, that is,
$J$ is an order ideal of $P$.  Thus vertices of 
$\Gamma_P$ are indexed by the connected order ideals $J$
in $\Jconn(P)$, the same indexing set as the
vertices of $\Delta_P$.

The condition for a pair of connected order ideals $\{J_1,J_2\}$ 
to index an edge in the nested set complex $\Delta_{\BBB(G)}$
is that they intersect trivially (either disjointly or nested) 
and if disjoint then they furthermore have
$G|_{J_1 \cup J_2}$ not in $\BBB(G)$, so that $J_1 \cup J_2$ is not
a connected order ideal.  But it is impossible for
two order ideals $J_1, J_2$ of $P$ to be disjoint
and have $J_1 \cup J_2$ a connected ideal:  this
would imply that there is some Hasse diagram edge
connecting them, giving
an order relation between some pair of elements $\{j_1,j_2\}$
with $j_i$ in $J_i$ for $i=1,2$, and would force either $j_1$ or $j_2$ to lie in
the intersection $J_1 \cap J_2$.  Thus $\{J_1,J_2\}$ index an edge of $\Gamma_P$
if and only if they intersect trivially, that is,
if and only if they index an edge of $\Delta_P$.  Hence $\Delta_P$
and $\Gamma_P$ are isomorphic flag complexes.
\end{proof}

The maximal cones in $\PPP_{\BBB(G)}$ correspond to what were called
{\it $\BBB(G)$-trees} in \cite[\S 7]{Postnikov} and \cite[\S 8.1]{PostnikovRWilliams}.
This means that the maximal simplices of the triangulation $\Delta_P$
will correspond to what we might call {\it $P$-forests}:
forest posets $F$ in which every principal ideal $F_{\leq i}$ 
is a connected order ideal of $P$, and
whenever $i,j$ are incomparable in the poset $F$, one has that the
ideal $F_{\leq i} \cup F_{\leq j}$ of $P$ is disconnected.

\begin{example}
For the poset $P$ on $\{1,2,3,4\}$ in which $1,3<_P 2,4$,
the Hasse diagram is the graph $G$ shown in Figure~\ref{normal-fan-figure}.
The acyclic orientation $\omega$ of $G$ whose transitive closure gives
$P$ corresponds to a quadrangular cone $\NNN_\omega$ which is the lowest on the page among
the three quadrangular cones depicted in $\NNN(\ZZZ_G)$.
This cone $\NNN_\omega$ is subdivided 
into four cones in $\NNN(\ZZZ_{K_4})$, corresponding
to the set of linear extensions $\LLL(P)=\{1324,1342,3124,3142\}$.
On the other hand, the cone $\NNN_\omega$ is subdivided 
into only two cones in $\NNN(\PPP_{\BBB(G)})$,
labelled in the figure by the two $\BBB(G)$-trees $1,3 < 2 < 4$
and $1,3 < 4 < 2$.
\end{example}

Note that unlike the usual triangulation of the cone
$\NNN_\omega$ of $P$-partitions
corresponding to the order complex $\Delta \JJJ(P)$ that was discussed in 
Section~\ref{sec:traditional-viewpoint},
the  maximal cones in the triangulation $\Delta_P$ 
are not unimodular.  In fact, each such maximal cone
corresponding to some $P$-forest $F$
will decompose into $|\LLL(F)|$ different unimodular
cones from the triangulation by $\Delta \JJJ(P)$, that is,
from the normal fan $\NNN(\ZZZ_{K_n})$ of the permutohedron.

\section{Other questions}
\label{questions-section}

We collect here some questions and problems left unresolved in this work.

\subsection{Resolving the rings $R_P$ over $S$
and Ferrers posets}

The following problem is motivated
by our desire to count linear extensions for more posets $P$.

\begin{problem}
Find more posets $P$ where one
can compute $\Hilb(R_P,\xx)$, possibly by
writing down an explicit $S$-resolution of $R_P$,
or $\gr(R_P)$ or $S/I^\init_p$.
\end{problem}

One particular instance originally motivated us, but has proven elusive
so far.  Given a number partition $\lambda$, consider the finite poset $P=P_\lambda$ 
on the set of squares $(i,j)$ in the Ferrers diagram for $\lambda$, 
partially ordered componentwise, with the square $(1,1)$ as maximum element.
Gansner \cite{Gansner} showed how the
Hillman-Grassl algorithm proves an interesting hook formula that
counts weak $P$-partitions $f$ by an intermediate multigrading, where one specializes
the variable $x_{i,j}$ associated with square $(i,j)$ to the variable $y_{i-j}$
recording its {\it content} $i-j$:

\begin{equation}
\label{colored-hook-formula}
\sum_{f \in \AAA^\weak(P_\lambda)} \prod_{(i,j) \in \lambda} y_{i-j}^{f(i,j)} 
= \prod_{(i,j) \in \lambda } \left( 1-\prod_{(i',j') \in H(i,j)} y_{i'-j'} \right)^{-1} .
\end{equation}
where here $H(i,j)$ denotes the set of squares of $\lambda$ lying in the
hook of square $(i,j)$.

\begin{question}
For these posets $P=P_\lambda$,
can we explain \eqref{colored-hook-formula} via an analysis of the 
structure of the ring $R_P$, or $\gr(R_P)$ or $S/I^\init_P$
that leads to its Hilbert series?  
Is one of these rings easy to resolve over $S$, for example?
\end{question}

\subsection{Further structure for the ideal $\III_P$ of $P$-partitions}

It can be shown (e.g., using \cite[Proposition 3]{MillerR}) that, 
for any poset $P$ on $\{1,2,\ldots,n\}$,
the ideal $\III(P)$ of $P$-partitions is a Cohen-Macaulay module,
either over the ring $R_P$ of weak $P$-partitions, or over
the polynomial algebra $S=k[U_J]_{J \in \Jconn(P)}$.
This raises several related questions about the modules $\III(P)$, 
beginning with the issue of their minimal generating sets, raised
in Remark~\ref{rmk:non-minimal-generators-remark}.

\begin{problem}
Describe the minimal monomial generators for $\III(P)$
over $R_P$.
\end{problem}

\noindent
Beyond minimal generating sets, one ultimately wants the following.
\begin{problem}
Given any poset $P$ on $\{1,2,\ldots,n\}$, describe for  $\III(P)$
\begin{enumerate}
\item[(i)]
an explicit resolution of $\III_P$ as an $S$-module or an
$R_P$-module, or both, and 
\item[(ii)]
the multigraded Betti numbers in the {\it minimal} free 
resolutions, that is, the multigraded vector spaces
$
\Tor^{S}_*(\III_P,k).
$
and
$
\Tor^{R_P}_*(\III_P,k).
$
\end{enumerate}
\end{problem}
Of course, there are similar questions one can
ask about the associated graded ring $\gr(R_P)$ and
associated graded modules $\gr(\III_P)$ over it,
and over $S$.

\begin{example}
Consider the poset $P=P_2$ from Example~\ref{ex:three-little-posets}, 
having order relations $2 <_P 1,3$.  Then $S=k[U_2,U_{12},U_{23},U_{123}]$, and
\begin{equation}
\label{example-S-hilbert-series}
\begin{aligned}
\Hilb(S,\xx) &=\frac{1}
      {(1-x_2)(1-x_1 x_2)(1-x_2 x_3)(1-x_1 x_2 x_3)} \\
\Hilb(S,q) &=\frac{1}
      {(1-q)(1-q^2)^2(1-q^3)} .
\end{aligned}
\end{equation}
It turns out that the generating set $\{x_2, x_2 x_3\}$
described in \eqref{non-minimal-generators} for the 
ideal $\III_P$ {\it is} minimal in this case, leading to the following
minimal free $S$-resolution
$$
\begin{array}{rccccl}
              &S(-(0,2,1)) &                 & S(-(0,1,0))&                &        \\
0 \rightarrow &\oplus       & \overset{A}{\longrightarrow} & \oplus     &\longrightarrow & \III_P   \\
              & S(-(1,2,1)) &              & S(-(0,1,1))&                &        \\
              &             &                & e_2        &\mapsto         &x_2  \\
              &             &                & e_{23}      &\mapsto         &x_2 x_3 \\
\end{array}
$$
where
$$
A=
\left[\begin{matrix}
U_{23} & -U_{123} \\
-U_2  & U_{12}
\end{matrix} \right].
$$
Together with the Hilbert series for $S$
given in \eqref{example-S-hilbert-series}, this allows one to
calculate 
$$
\begin{aligned}
\Hilb(\III_P,\xx) &= \frac{x_2+x_2x_3 - x_2^2x_3 - x_1x_2^2x_3 }
                        {(1-x_2)(1-x_1 x_2)(1-x_2 x_3)(1-x_1 x_2 x_3)} \\
\Hilb(\III_P,q) &= \frac{q+q^2 - (q^3+q^4)}{ (1-q)(1-q^2)^2(1-q^3) } 
                = \frac{q+q^2} { (1-q)(1-q^2)(1-q^3) } . \\
\sum_{w \in \LLL(P)} q^{\maj(w)} &= q^2 + q^3.
\end{aligned}
$$

\end{example}

Lastly, given Stanley's characterization for when $R_P$ is
Gorenstein discussed in Section~\ref{non-natural-labelings-section} 
above, it is reasonable to ask the following.
\begin{problem}
Characterize when $\III_P$ is Gorenstein, that is, when
one has an isomorphism $\Omega(\III_P) \cong \III_P$, up to a
shift in grading. 
\end{problem}
\noindent
This should be approachable, as the
canonical module $\Omega(\III_P)$ has a simple
description (via \cite[Proposition 3]{MillerR}):  
it is the ideal within $R_P$ spanned $k$-linearly by
the monomials $\xx^f$ as $f$ runs through
those weak $P$-partitions $f:P \rightarrow \NN$ for which
$$
\begin{aligned}
f(i) \geq_\NN f(j) &\text{ if }i \lessdot_P j\\
f(i) >_\NN f(j)    &\text{ if }i <_\NN j.
\end{aligned}
$$

\section*{Acknowledgements}
The authors thank C.E. Csar and Steven Sam for
helpful suggestions.


\end{document}